\documentclass[a4paper,11pt]{article}
\usepackage[latin1]{inputenc}
\usepackage[english]{babel}
\usepackage{amsmath}
\usepackage{amsfonts}
\usepackage{amssymb}
\usepackage{epsfig}
\usepackage{amsopn}
\usepackage{amsthm}
\usepackage{color}
\usepackage{graphicx}
\usepackage{enumerate}
\newtheorem{theorem}{Theorem}[section]
\newtheorem{corollary}[theorem]{Corollary}
\newtheorem{lemma}[theorem]{Lemma}
\newtheorem{proposition}[theorem]{Proposition}
\newtheorem{definition}[theorem]{Definition}

\newtheorem*{theorem*}{Theorem}
\newtheorem*{lemma*}{Lemma}
\newtheorem*{remark*}{Remark}
\newtheorem*{definition*}{Definition}
\newtheorem*{proposition*}{Proposition}
\newtheorem*{corollary*}{Corollary}
\numberwithin{equation}{section}
\newcommand{\real}{\mathbb{R}}

\let\ced=\c         

\def\a{\alpha}
\def\b{\beta}

\def\e{\varepsilon}        

\def\ca{{\cal A}}

\def\cf{{\cal F}}

\def\cm{{\cal M}}

\def\cs{{\cal S}}

\def\qed{\,\unskip\kern 6pt \penalty 500
\raise -2pt\hbox{\vrule \vbox to8pt{\hrule width 6pt
\vfill\hrule}\vrule}\par}
\definecolor{darkblue}{rgb}{0.05, .05, .65}
\definecolor{darkgreen}{rgb}{0.1, .65, .1}
\definecolor{darkred}{rgb}{0.8,0,0}
\newcommand{\beqn}{\begin{equation}}
\newcommand{\eeqn}{\end{equation}}
\newcommand{\bear}{\begin{eqnarray}}
\newcommand{\eear}{\end{eqnarray}}
\newcommand{\bean}{\begin{eqnarray*}}
\newcommand{\eean}{\end{eqnarray*}}
\begin{document}

\title{\huge \bf Asymptotic behavior for a singular diffusion equation with gradient absorption}

\author{
\Large Razvan Gabriel Iagar\,\footnote{Departamento de Análisis
Matemático, Univ. de Valencia, Dr. Moliner 50, 46100, Burjassot
(Valencia), Spain, \textit{e-mail:}
razvan.iagar@uv.es},\footnote{Institute of Mathematics of the
Romanian Academy, P.O. Box 1-764, RO-014700, Bucharest, Romania.}
\\[4pt] \Large Philippe Lauren\c cot\,\footnote{Institut de
Math\'ematiques de Toulouse, CNRS UMR~5219, Universit\'e de
Toulouse, F--31062 Toulouse Cedex 9, France. \textit{e-mail:}
Philippe.Laurencot@math.univ-toulouse.fr}\\ [4pt] }
\date{\today}
\maketitle

\begin{abstract}
We study the large time behavior of non-negative solutions to the singular diffusion equation with gradient absorption
$$
\partial_t u-\Delta_{p}u+|\nabla u|^q=0 \quad \hbox{in} \
(0,\infty)\times\real^N,
$$
for $p_c:=2N/(N+1)<p<2$ and $p/2<q<q_*:=p-N/(N+1)$. We prove that
there exists a unique very singular solution of the equation, which
has self-similar form and we show the convergence of general
solutions with suitable initial data towards this unique very
singular solution.
\end{abstract}

\vspace{2.0 cm}

\noindent {\bf AMS Subject Classification:} 35K67, 35K92, 35B40,
35D40.

\medskip

\noindent {\bf Keywords:} large time behavior, singular diffusion, gradient absorption, very singular solutions, $p$-Laplacian, bounded measures.

\newpage

\section{Introduction and results}\label{sec1}

The aim of the present paper is to study the large time behavior of non-negative solutions to the following equation with singular diffusion and gradient absorption:
\begin{equation}\label{eq1}
\partial_{t}u-\Delta_{p}u+|\nabla u|^q=0, \quad
(t,x)\in Q_\infty := (0,\infty)\times\real^N,
\end{equation}
for $p_c:=2N/(N+1)<p<2$ and $p/2<q<q_*:=p-N/(N+1)$. We consider only non-negative initial data
\begin{equation}
u(0,x)=u_0(x), \quad x\in\real^N, \label{inco}
\end{equation}
under suitable decay and regularity assumptions that will be specified later.
Equation \eqref{eq1} presents a competition between the effects of
the two terms: one term of singular diffusion $\Delta_p u := \text{div}\left( |\nabla u|^{p-2} \nabla u \right)$, which in our case is supercritical (that is, $p>p_c=2N/(N+1)$) in order to avoid extinction in finite time, and another term of nonlinear absorption depending on the gradient $|\nabla u|^q$. Due to this competition, interesting mathematical features appear in some ranges of exponents $p$ and $q$.

The qualitative theory of \eqref{eq1} for general exponents $p$ and
$q$ developed very recently; indeed, while there are many (even
classical ones) papers on nonlinear diffusion equations with zero
order absorption, covering almost all possible cases, the study of
the gradient absorption proved to be much more involved and brought
a bunch of very interesting mathematical phenomena, some of them
having been the subject of intensive research in the last decade. As
expected, the first results were obtained in the semilinear case
$p=2$, where the asymptotic behavior for $q>1$ has been identified
in a series of papers \cite{BKaL04, BKL04, BL01, BVDxx, BGK04, GL07,
Gi05}. Finite time extinction was shown to take place for $q\in
(0,1)$ \cite{BLS01, BLSS02, Gi05} while the critical case $q=1$, in
spite of its apparent simplicity, is still far from being fully
understood: only some large-time estimates are available
\cite{BRV97} but no precise asymptotics. Passing to the
$p$-Laplacian is a natural step, and for the slow-diffusion case
$p>2$, the exponent $q=p-1$ proved to have a very interesting
critical effect, as an interface between absorption-dominated
behavior and diffusion-dominated behavior \cite{BtL08, LV07}, while
itself gives rise to a critical regularized sandpile-type behavior,
as shown recently in \cite{ILV}. A natural next step was then to
pass to the study of the fast-diffusion case $1<p<2$, where the
authors made important progress recently in understanding the decay
rates and typical self-similar profiles \cite{IL1, IL2}. In
particular, finite time extinction was shown to take place when
$(p,q)$ ranges in $(p_c,2)\times (0,p/2)$ and in $(1,p_c)\times
(0,\infty)$ while diffusion is likely to govern the large time
dynamics when $(p,q)\in (p_c,2)\times (q_*,\infty)$. The
intermediate range $(p,q)\in (p_c,2)\times (p/2,q_*)$ features a
balance between the diffusion and absorption terms and is the focus
of this paper.

From now on, we restrict ourselves to the following range of
exponents:
\begin{equation}
p\in (p_c,2) \;\;\text{ and }\;\; q\in \left( \frac{p}{2} , q_* \right)\,, \label{rexp}
\end{equation}
and we set
\begin{equation}
\alpha := \frac{p-q}{2q-p}>0\,, \quad \beta := \frac{q-p+1}{2q-p}>0  \;\;\text{ and }\;\; \eta := \frac{1}{N(p-2)+p} > 0\,, \label{expas}
\end{equation}
the positivity of $\eta$ being a consequence of $p>p_c$. We also observe that, thanks to \eqref{rexp},
\begin{equation}
\alpha - N \beta = \frac{(N+1)(q_*-q)}{2q-p}>0\,. \label{expas2}
\end{equation}

In order to state the main result concerning the large-time
behavior, we recall a special category of solutions to \eqref{eq1},
that are called \emph{very singular solutions}. These are solutions
to \eqref{eq1} with an initial trace at $t=0$ more concentrated at
the origin than a Dirac mass, thus justifying the name. The precise
definition is given in Definition~\ref{def.VSS} at the beginning of
Section~\ref{sec4}.

The name \emph{very singular solution} has been introduced in
\cite{BPT86} for the heat equation with absorption of order zero.
After this first paper, many other very singular solutions for
diffusion equations with absorption terms were constructed, see
\cite{CQW07, KV92, Le96, PW88, Sh04, Zh94} and the references
therein. For \eqref{eq1}, we have established in \cite[Theorem
1.1]{IL2} the existence and uniqueness of such a very singular
solution to \eqref{eq1}, under the more restrictive hypothesis of
radial symmetry and self-similarity. We recall this result for the
reader's convenience as Theorem~\ref{th.VSSunique} below. For the
moment, let us denote this unique radially symmetric, self-similar
very singular solution by $U$ with
\begin{equation}\label{eq.selfVSS}
U(t,x):=t^{-\a}f_U(xt^{-\b}) , \qquad (t,x)\in Q_\infty .
\end{equation}
The main result about large time behavior is the following:

\begin{theorem}\label{th.asympt}
Let $u_0$ be a function such that
\begin{equation}
u_0\in L^1(\real^N)\cap W^{1,\infty}(\real^N)\,, \quad u_0\ge 0\,,
\quad u_0\not\equiv 0\,. \label{wp1}
\end{equation}
and
\begin{equation}\label{wp111}
\lim\limits_{|x|\to\infty}|x|^{\a/\b}u_0(x)=0.
\end{equation}
Then, the following large time behavior holds true:
\begin{equation}\label{asympt}
\lim\limits_{t\to\infty}t^{\a} \|u(t)-U(t)\|_\infty=0,
\end{equation}
where $U$ is the unique radially symmetric
self-similar very singular solution to \eqref{eq1} introduced in
\eqref{eq.selfVSS}.
\end{theorem}

In order to prove Theorem~\ref{th.asympt}, several steps are needed,
some of them being also very interesting by themselves. A very
important element of the proof is identifying the possible limits as
$t\to\infty$, that we can prove to be very singular solutions in the
sense of Definition~\ref{def.VSS} by viscosity techniques. Thus, the
circle will be closed by the following general uniqueness result.

\begin{theorem}\label{th.VSSgeneral}
There exists a unique very singular solution to \eqref{eq1} in the sense of Definition~\ref{def.VSS}. In particular, this
solution is radially symmetric and in self-similar form and it
coincides with $U$.
\end{theorem}

This theorem is an important extension of \cite[Theorem 1.1]{IL2},
where the uniqueness of a very singular solution is established
under the extra conditions of radial symmetry and self-similar form.
An interesting by-product of Theorem~\ref{th.VSSgeneral} is a
comparison principle for the elliptic equation
$$
-\Delta_{p}v+|\nabla v|^q-\a v-\b x\cdot\nabla v=0, \quad x\in
\real^N,
$$
under suitable conditions as $|x|\to\infty$. For a precise form of
the statement, we refer the reader to Theorem~\ref{th.comp} below.

On the way to proving Theorem~\ref{th.VSSgeneral}, we found out that
a theory of the Cauchy problem associated to \eqref{eq1} with
non-negative and bounded measures as initial data had to developed.
We thus prove an interesting result of well-posedness for
\eqref{eq1} for such initial data which extends to $p\in (p_c,2)$
the existing one for the semilinear case $p=2$ \cite{BL99, BVDxx}
but holds true only if the singular diffusion equation $\partial_t v
- \Delta_p v = 0$ in $Q_\infty$ is well-posed in this setting.
However, this issue seems to be still an open question for general
non-negative and bounded measures but the answer is positive for
Dirac masses which is exactly what is needed for the proof of
Theorem~\ref{th.VSSgeneral}.

\begin{theorem}\label{th.uniqfund}
Consider a non-negative bounded Borel measure
$u_0\in\cm_{b}^{+}(\real^N)$. If the singular diffusion equation
\begin{eqnarray*}
\partial_t v - \Delta_p v & = & 0 \;\;\text{ in }\;\; Q_\infty\,, \\
v(0) & = & u_0 \;\;\text{ in }\;\; \real^N\,,
\end{eqnarray*}
has a unique solution $v\in C([0,\infty); \cm_{b}^{+}(\real^N)) \cap
C(Q_\infty)$, then there exists a unique non-negative function $u\in
C(Q_\infty)$ which is a viscosity solution to \eqref{eq1} in
$Q_\infty$ and satisfies
\begin{equation}\label{cp.fund}
\lim_{t\to 0} \int_{\real^N} \psi(x)\ u(t,x)\, dx = \int_{\real^N}
\psi(x)\ du_0(x)
\end{equation}
for any bounded and continuous function $\psi\in BC(\real^N)$.
Moreover, $u(t)$ belongs to $L^1(\real^N)\cap W^{1,\infty}(\real^N)$
for all $t>0$ and satisfies
\begin{equation}
\|u(t)\|_1 \le M_0 := \int_{\real^N} du_0(x)\,, \label{evian}
\end{equation}
as well as the following estimates
\begin{equation}\label{wp101}
\|u(t)\|_{1} + t^{N\eta}\|u(t)\|_{\infty} \leq C_s(M_0)\,,
\end{equation}
and
\begin{equation}\label{wp102}
\|\nabla u(t)\|_{\infty}\leq C_s(M_0) \left( 1 +
t^{(N+1)(q_*-q)\eta/(p-q)} \right) t^{-(N+1)\eta}\,,
\end{equation}
where $C_s\in C([0,\infty))$ is a positive function depending only on
$N$, $p$, and $q$.
\end{theorem}

The proof of this theorem is technical and quite involved, as usual
when dealing with measures, since the lack of regularity does not
allow to apply some of the standard techniques. In particular,
Theorem~\ref{th.uniqfund} also implies the existence and uniqueness
of a fundamental solution with any given mass $M>0$ to \eqref{eq1},
as it is explained at the end of Section~\ref{sec3}.

\medskip

\noindent \textbf{Organisation of the paper.} We collect in Section~\ref{sec2} many technical results and estimates
needed in the sequel, in the form of separate lemmas. These include:
a rigorous definition of viscosity solutions, decay estimates,
estimates on the tail of the solution at sufficiently large times,
and estimates of the solutions for small times, which are useful
tools for identifying the initial trace. We agree that this section
is a bit technical, but this allows us to state more clearly the
main ideas and steps in the proofs of our main results. A reader who
is not so interested in technical details could skip this part and
admit the technical lemmas, or come back to it later.

Section~\ref{sec3} is devoted to the proof of
Theorem~\ref{th.uniqfund}. The proof is divided into two steps: we
first construct a solution to \eqref{eq1} by classical approximation
arguments. We next proceed to show the uniqueness of the solution
which is actually the main contribution of this section. We then
pass to the proof of Theorem~\ref{th.VSSgeneral}, which occupies
almost all Section~\ref{sec4} and is divided into several steps: we
first construct a maximal and a minimal element in the class of the
very singular solutions to \eqref{eq1}. Then, we find that these two
solutions are identical, by identifying both of them with the unique
radially symmetric and self-similar very singular solution $U$, and
we end up with the proof of the comparison principle for the
associated elliptic equation. We end the paper with the proof of
Theorem~\ref{th.asympt}, to which Section~\ref{sec5} is devoted. It
relies on the half-relaxed limits technique and is rather short,
since most of the needed technical facts were already done in
previous sections.

\section{Well-posedness and decay estimates}\label{sec2}

In this section, we collect previous results on the well-posedness
of \eqref{eq1} as well as some qualitative properties of the
solutions. Let us first recall the notion of solutions we use
throughout the paper.

\subsection{Viscosity solution}\label{sec2vs}

As in our previous works \cite{IL1, IL2}, a suitable notion of solution for equation~\eqref{eq1} is that of \emph{viscosity solution}, which is useful in
dealing with the gradient term. Due to the singular character of
\eqref{eq1} at points where $\nabla u$ vanishes, the standard
definition of viscosity solution has to be adapted to deal with this
case \cite{IS, JLM, OS}. In fact, it requires to restrict the class
of comparison functions \cite{IS, OS}. More precisely, let $\cf$ be
the set of functions $f\in C^{2}([0,\infty))$ satisfying
$$
f(0)=f'(0)=f''(0)=0, \ f''(r)>0 \ \hbox{for} \ \hbox{all} \ r>0,
\quad \lim\limits_{r\to 0}|f'(r)|^{p-2}f''(r)=0.
$$
For example, $f(r)=r^{\sigma}$ with $\sigma>p/(p-1)>2$ belongs to
$\cf$. We then introduce the class $\ca$ of admissible comparison
functions $\psi$ defined as follows: a function $\psi\in
C^2(Q_{\infty})$ belongs to $\ca$ if, for any $(t_0,x_0)\in
Q_{\infty}$ where $\nabla\psi(t_0,x_0) =0$, there exist a constant
$\delta>0$, a function $f\in\cf$, and a modulus of continuity
$\omega\in C([0,\infty))$, (that is, a non-negative function
satisfying $\omega(r)/r\to 0$ as $r\to 0$), such that, for all
$(t,x)\in Q_{\infty}$ with $|x-x_0|+|t-t_0|<\delta$, we have
$$
|\psi(t,x)-\psi(t_0,x_0)-\partial_t\psi(t_0,x_0)(t-t_0)|\le
f(|x-x_0|)+\omega(|t-t_0|).
$$
With these notations, viscosity solutions to \eqref{eq1} are defined as follows \cite{IS, JLM, OS}:

\begin{definition}\label{def.visc}
An upper semicontinuous function $u:Q_{\infty}\to\real$ is a
viscosity subsolution to \eqref{eq1} in $Q_{\infty}$ if, whenever
$\psi\in\ca$ and $(t_0,x_0)\in Q_{\infty}$ are such that
\begin{equation*}
u(t_0,x_0)=\psi(t_0,x_0), \quad u(t,x)<\psi(t,x) \ \mbox{for all}\
(t,x)\in Q_{\infty}\setminus\{(t_0,x_0)\},
\end{equation*}
then
\begin{equation*}
\left\{\begin{array}{ll}\partial_t\psi(t_0,x_0)\leq\Delta_{p}\psi(t_0,x_0)-|\nabla\psi(t_0,x_0)|^{q}
&
\ \hbox{if} \ \nabla\psi(t_0,x_0)\neq0,\\
\partial_t \psi(t_0,x_0)\leq 0 & \ \hbox{if} \
\nabla\psi(t_0,x_0)=0.\end{array}\right.
\end{equation*}
A lower semicontinuous function $u:Q_{\infty}\to\real$ is a
viscosity supersolution to \eqref{eq1} in $Q_{\infty}$ if, whenever
$\psi\in\ca$ and $(t_0,x_0)\in Q_{\infty}$ are such that
\begin{equation*}
u(t_0,x_0)=\psi(t_0,x_0), \quad u(t,x)>\psi(t,x) \ \mbox{for all}\
(t,x)\in Q_{\infty}\setminus\{(t_0,x_0)\},
\end{equation*}
then
\begin{equation*}
\left\{\begin{array}{ll}\partial_t\psi(t_0,x_0)\geq\Delta_{p}\psi(t_0,x_0)-|\nabla\psi(t_0,x_0)|^{q}
&
\ \hbox{if} \ \nabla\psi(t_0,x_0)\neq0,\\
\partial_t \psi(t_0,x_0)\geq 0 & \ \hbox{if} \
\nabla\psi(t_0,x_0)=0.\end{array}\right.
\end{equation*}

A continuous function $u:Q_{\infty}\to\real$ is a viscosity solution to \eqref{eq1} in $Q_{\infty}$ if it is a viscosity subsolution and supersolution.
\end{definition}

A remarkable feature of this modified definition is that basic
results about viscosity solutions, such as comparison principle and stability property, are still valid, see \cite[Theorem 3.9]{OS}
(comparison principle) and \cite[Theorem 6.1]{OS} (stability). The
relationship between viscosity solutions and other notions of
solutions is investigated in \cite{JLM}. From now on, by a solution
to \eqref{eq1} we mean a viscosity solution in the sense of
Definition~\ref{def.visc} above.

With this notion of solution to \eqref{eq1}, we have the following well-posedness result \cite[Theorem~6.2]{IL1}.

\begin{proposition}\label{pr.wp1}
Assume that $u_0$ is a function satisfying the conditions
\eqref{wp1}. Then there exists a unique non-negative function $u\in
C([0,\infty)\times\real^N)$ which is a viscosity solution to
\eqref{eq1} in $Q_\infty$ and satisfies $u(0)=u_0$. In addition,
$u(t)\in L^1(\real^N)\cap W^{1,\infty}(\real^N)$ for each $t>0$ and
$u$ is also a weak solution to \eqref{eq1}-\eqref{inco} in the
following sense:
\begin{equation}\label{def.weak}
\int_{\real^N} (u(t,x) -u(s,x)) \psi(x) \,dx + \int_s^t\int_{\real^N} \left(|\nabla u|^{p-2} \nabla u \cdot \nabla\psi + |\nabla u|^q \psi \right)\,dx\,d\tau=0,
\end{equation}
for any $0\leq s<t<\infty$ and $\psi\in C_0^{\infty}(\real^N)$.
\end{proposition}

As usual for homogeneous parabolic equations, the radial symmetry and monotonicity are preserved, as the following result states.

\begin{lemma}\label{lem.rs}
If $u_0$ satisfies \eqref{wp1} and is radially symmetric and
non-increasing with respect to $|x|$, then the same properties hold true for $u(t)$, for any $t>0$.
\end{lemma}

\begin{proof}
The radial symmetry of $u(t)$ for positive times $t>0$ follows readily from the rotational invariance of \eqref{eq1} and the well-posedness of \eqref{eq1}. Next, we can write
$u(t,x)=u(t,|x|)=u(t,r)$, and it satisfies
$$
\partial_{t}u-(p-1)|\partial_{r}u|^{p-2}\partial^2_{r}u-\frac{N-1}{r}|\partial_{r}u|^{p-2}\partial_{r}u+|\partial_{r}u|^q=0.
$$
At a formal level, it is clear that the zero function is a solution
to the equation satisfied by $\partial_{r}u$ (which can be derived
by differentiating the above equation for $u$), and the claimed
monotonicity follows from the comparison principle since $\partial_r
u_0\leq0$. Thanks to the uniqueness of solutions to \eqref{eq1},
this argument can be made rigorous by standard approximations, as in
\cite{IL1}.
\end{proof}

A classical property of parabolic equations is that a modulus of
continuity in space entails a modulus of continuity in time. In that
direction, we have the following result which can be proved as
\cite[Lemma~5]{GGK03}.

\begin{lemma}\label{le.wp5}
Consider an initial condition $u_0$ satisfying \eqref{wp1} and let
$u$ be the corresponding solution to \eqref{eq1}-\eqref{inco}.
Assume further that there are $\tau\ge 0$ and $A>0$ such that
$\|\nabla u(t)\|_\infty \le A$ for all $t\in [\tau,\infty)$. Then
there is $C_2>0$ depending only on $N$, $p$, and $q$ such that
\begin{equation}
|u(t,x) - u(s,x)| \le C_2\ \left[ (1+A)\ |t-s|^{1/2} + A^q\ |t-s|
\right]\,, \qquad t>s\ge \tau\,. \label{wp8}
\end{equation}
\end{lemma}

\subsection{Decay estimates}\label{sec2de}

We next recall temporal decay estimates in $L^1(\real^N)$ and $W^{1,\infty}(\real^N)$ which are consequences of the analysis performed in \cite{IL1} and depend on the behavior of the initial data as $|x|\to\infty$.

\begin{proposition}\label{pr.wp2a}
Assume that $u_0$ satisfies \eqref{wp1} and denote the corresponding solution to \eqref{eq1}-\eqref{inco} by $u$. Then there is a constant $C>0$ depending only on $N$, $p$, and $q$ such that
\begin{equation}
|\nabla u(t,x)| \le C \left( \|u(s)\|_\infty^{1/\alpha p} + (t-s)^{-1/p} \right) \left( u(t,x) \right)^{2/p}\,, \quad 0\le s<t\,, \;\; x\in\real^N\,. \label{wp100}
\end{equation}
In addition, if $M$ is such that $M\ge \|u_0\|_1$, then the
estimates \eqref{wp101} and \eqref{wp102} hold true with $C_s(M)$
instead of $C_s(M_0)$.
\end{proposition}

\begin{proof}
The estimate \eqref{wp100} is a straightforward consequence of
\cite[Theorem~1.3~(i) \&~(ii)]{IL1}, while \eqref{wp101} follows by
comparison with the solution $v$ to the diffusion equation
\begin{eqnarray}
\partial_t v - \Delta_p v & = & 0 \;\;\text{ in }\;\; Q_\infty\,, \label{wp103} \\
v(0) & = & u_0 \;\;\text{ in }\;\; \real^N\,, \label{wp104}
\end{eqnarray}
see \cite{DBH90} for instance. Indeed, we obviously have $u\le v$ in $Q_\infty$ by the comparison principle and, since $p>p_c$, we deduce from \cite[Lemma~III.6.1 \& Theorem~III.6.2]{DBH90} (with $r=1$ and $R=\infty$) that
\begin{equation}
\|v(t)\|_1 \le C\ \|u_0\|_1 \;\;\text{ and }\;\; \|v(t)\|_\infty \le C\ \|u_0\|_1^{p\eta}\ t^{-N\eta} \label{wp300}
\end{equation}
for $t>0$. Finally, \eqref{wp102} readily follows from \eqref{wp100} (with $s=t/2$) and \eqref{wp101}.
\end{proof}

\medskip

For initial data decaying sufficiently fast as $|x|\to\infty$, faster temporal decay estimates were also supplied in \cite[Theorem~1.2]{IL1}, which are only valid when $p$ and $q$ satisfy \eqref{rexp}.

\begin{proposition}\label{pr.wp2}
Assume that $u_0$ satisfies \eqref{wp1} as well as
\begin{equation}
0 \le u_0(x) \le \kappa\ |x|^{-\alpha/\beta}\,, \quad x\in\real^N\,, \label{wp2}
\end{equation}
for some $\kappa>0$, and denote the corresponding solution to
\eqref{eq1}-\eqref{inco} by $u$. Then there is a constant
$K_\kappa>0$ depending only on $N$, $p$, $q$, and $\kappa$ such that
\begin{equation}\label{wp3}
t^{\alpha-N \beta}\|u(t)\|_{1} + t^{\alpha}\|u(t)\|_{\infty} + t^{\alpha+\beta} \|\nabla u(t)\|_{\infty} \leq K_\kappa, \quad t>0.
\end{equation}
\end{proposition}

The precise dependence of $K_\kappa$ on the parameters is not stated in \cite[Theorem~1.2~(i)]{IL1} but can be recovered by inspecting the proofs of \cite[Theorem~1.2~(i) \& Lemma~5.1]{IL1}.

\subsection{Small time estimates}\label{sec2ste}

The previous decay estimates allow us to analyze precisely the
behavior of solutions to \eqref{eq1} for small times, a fact which
will be of utmost importance when considering non-smooth or even
singular initial data.

\begin{proposition}\label{pr.wp2b}
Assume that $u_0$ satisfies \eqref{wp1} and denote the corresponding solution to \eqref{eq1}-\eqref{inco} by $u$.
\begin{itemize}
\item[(a)]  Let $\psi\in C_0^\infty(\real^N)$ and $T>0$. If $M$ is such that $M\ge \|u_0\|_1$, there exists a constant $C(M,T)>0$ depending only on $N$, $p$, $q$, $M$, and $T$ such that, for $t\in (0,T)$,
\begin{equation}
\begin{split}
& \left| \int_{\real^N} (u(t,x) - u_0(x))\ \psi(x)\, dx \right| \\
\le & C(M,T) \left[ \|\psi\|_\infty\ t^{(N+1)(q_*-q)\eta} + \|\nabla\psi\|_{p/(2-p)}\ t^{1/p} \right]\,.
\end{split} \label{wp200}
\end{equation}
\item[(b)] Let $\psi\in C_0^\infty(\real^N)$ be a non-negative function such that $\psi(x)=0$ for $x\in B_r(0)$ for some $r>0$. If $u_0$ satisfies \eqref{wp2} for some $\kappa>0$, there exists a constant $C(\kappa,r)>0$ depending only on $N$, $p$, $q$, $\kappa$, and $r$ such that, for $t>0$,
\begin{equation}
\int_{\real^N} u(t,x)\ \psi(x)\, dx \le \int_{\real^N} u_0(x)\ \psi(x)\, dx + C(\kappa,r)\|\nabla\psi\|_{p/(2-p)}\ t^{1/p} \,. \label{wp201}
\end{equation}
\end{itemize}
\end{proposition}

\begin{proof}
\textbf{Case~(a).} Let $\psi\in C_0^\infty(\real^N)$, $T>0$, and $t\in (0,T)$. It follows from \eqref{def.weak} that
\begin{equation}\label{volvic}
\begin{split}
& \left| \int_{\real^N} (u(t,x) - u_0(x))\ \psi(x)\,dx \right| \\
\leq &\int_0^t \int_{\real^N} \left( |\nabla u(s,x)|^{p-1}\
|\nabla\psi(x)| + |\nabla u(s,x)|^{q}\ \psi(x)\right)\,dx\,ds.
\end{split}
\end{equation}
To estimate the gradient terms in the right-hand side of \eqref{volvic}, we first notice that \eqref{wp100} and \eqref{wp101} give for $(s,x)\in Q_\infty$
\begin{eqnarray}
\left| \nabla u(s,x) \right| & \le & C \left[ \left\| u\left( \frac{s}{2} \right) \right\|_\infty^{1/\alpha p} + s^{-1/p} \right] \left( u(s,x)
\right)^{2/p}\,, \nonumber\\
& \le & C(M) \left[ s^{-N\eta/\alpha p} + s^{-1/p} \right] \left( u(s,x)
\right)^{2/p}\,. \label{grad.est.u2}
\end{eqnarray}
Now, we infer from \eqref{wp101} and \eqref{grad.est.u2} that
\begin{align*}
 \int_{\real^N} |\nabla u(s,x)|^q\ |\psi(x)| \,dx
\leq & C(M)\ \|\psi\|_\infty \left[ s^{-qN\eta/\alpha p} + s^{-q/p} \right] \|u(s)\|_\infty^{(2q-p)/p}\ \|u(s)\|_1 \\
\leq & C(M)\ \|\psi\|_\infty \left[ s^{-N\eta/\alpha} + s^{-((N+1)q-N)\eta} \right].
\end{align*}
Observing that
\begin{align*}
& 1 - \frac{N\eta}{\alpha} = \frac{(N+1)(q_*-q)p\eta}{p-q} > 0,\\
& 1-((N+1)q-N)\eta = (N+1)(q_*-q)\eta > 0
\end{align*}
by \eqref{rexp}, we integrate the above inequality over
$(0,t)$ and obtain
\begin{align}
& \int_0^t\int_{\real^N} |\nabla u(s,x)|^{q} |\psi(x)| \,dx\,ds \nonumber \\
\leq & C(M)\ \|\psi\|_{\infty} \left[ t^{(N+1)(q_*-q)p\eta/(p-q)} + t^{(N+1)(q_*-q)\eta} \right] \nonumber \\
\leq & C(M)\ \|\psi\|_{\infty} \left[ 1 + t^{(N+1)(q_*-q)q\eta/(p-q)} \right] t^{(N+1)(q_*-q)\eta}. \label{interm3bis}
\end{align}
Similarly, by \eqref{grad.est.u2} and H\"older's inequality,
\begin{align*}
& \int_{\real^N} |\nabla u (s,x)|^{p-1}\ |\nabla\psi(x)|\,dx \\
&\leq C(M) \left[ s^{-(p-1)N\eta/\alpha p} + s^{-(p-1)/p} \right] \int_{\real^N} (u(s,x))^{2(p-1)/p}\ |\nabla\psi(x)|\,dx\\
&\leq
C(M) \left[ s^{-(p-1)N\eta/\alpha p} + s^{-(p-1)/p} \right] \| u(s)\|_1^{2(p-1)/p}\ \|\nabla\psi\|_{p/(2-p)}\\
&\leq C(M) \left[ 1 + s^{(p-1)(N+1)(q_*-q)\eta/(p-q)} \right] \|\nabla\psi\|_{p/(2-p)}\ s^{-(p-1)/p} ,
\end{align*}
hence, after integrating over $(0,t)$,
\begin{equation}\label{interm2bis}
\begin{split}
& \int_0^t\int_{\real^N} |\nabla us,x)|^{p-1}\ |\nabla\psi(x)|
\,dx\,ds \\
\leq & C(M) \left[ 1 + t^{(p-1)(N+1)(q_*-q)\eta/(p-q)} \right] \|\nabla\psi\|_{p/(2-p)}\ t^{1/p}.
\end{split}
\end{equation}
Combining \eqref{volvic}, \eqref{interm3bis}, and \eqref{interm2bis} gives \eqref{wp200}.

\smallskip

\noindent\textbf{Case~(b).} Let $t>0$ and a non-negative function $\psi\in C_0^\infty(\real^N)$. Since $u_0$ satisfies \eqref{wp2}, it follows from \eqref{wp100} and \eqref{wp3} that, for $(s,x)\in Q_\infty$,
\begin{eqnarray}
\left| \nabla u(s,x) \right| & \le & C \left[ \left\| u\left( \frac{s}{2} \right) \right\|_\infty^{1/\alpha p} + s^{-1/p} \right] \left( u(s,x)
\right)^{2/p}\,, \nonumber\\
& \le & C(\kappa)\ s^{-1/p} \left( u(s,x)
\right)^{2/p}\,. \label{grad.est.u2b}
\end{eqnarray}
Owing to the non-negativity of $\psi$, it follows from \eqref{def.weak} and \eqref{grad.est.u2b} that
\begin{align*}
\int_{\real^N} (u(t,x)-u_0(x))\ \psi(x)\, dx \le & \int_0^t \int_{\real^N} |\nabla u(s,x)|^{p-1}\ |\nabla\psi(x)|\, dx\, ds \\
\le & C(\kappa)\ \int_0^t \int_{\real^N} \left( u(s,x) \right)^{2(p-1)/p}\ |\nabla\psi(x)|\ s^{-(p-1)/p} \, dx\, ds\,.
\end{align*}
We now use again the decay property \eqref{wp2} of $u_0$ together with \cite[Equation~(5.5)]{IL1} to conclude that $u(s,x)\le C(\kappa)\ |x|^{-\alpha/\beta}$ for $(s,x)\in Q_\infty$. Since $\psi$ vanishes in $B_r(0)$ then so does $\nabla\psi$ and, by H\"older's inequality,
\begin{align*}
\int_{\real^N} (u(t,x)-u_0(x))\ \psi(x)\, dx \le & C(\kappa)\ \int_0^t \int_{\{|x|>r\}} |x|^{-2(p-1)\alpha/p\beta}\ |\nabla\psi(x)|\ s^{-(p-1)/p} \,dx\, ds \\
\le & C(\kappa)\ t^{1/p}\ \left( \int_{\{|x|>r\}} |x|^{-\alpha/\beta}\, dx \right)^{2(p-1)/p}\ \|\nabla\psi\|_{p/(2-p)}\,,
\end{align*}
from which \eqref{wp201} follows since $\alpha/\beta>N$ by \eqref{expas2}.
\end{proof}

\subsection{Tail behavior}\label{sec2tb}

We end this section with a control on the tail of solutions to
\eqref{eq1}-\eqref{inco}. We first establish a pointwise estimate by
showing the existence of a universal upper bound (also refered to as
a \emph{friendly giant} in literature), an idea also used in
previous works, see \cite{BKL04, BL01, KV92, VazquezPME} for
instance. We define
\begin{equation}\label{FG}
\Gamma_{p,q}(r):=\gamma\ r^{-\alpha/\beta}, \quad r>0,
\end{equation}
where
\begin{equation}\label{exp.FG}
\gamma:=\frac{q-p+1}{p-q}\ \left( \frac{p-1}{q-p+1} \right)^{1/(q-p+1)},
\end{equation}
and first state some useful properties of $\Gamma_{p,q}$.

\begin{lemma}\label{le.wp6}
For all $r>0$, $\Gamma_{p,q}$ belongs to $L^1(\real^N\setminus B_r(0))$ and $(t,x)\longmapsto \Gamma_{p,q}(|x|-r)$ is a supersolution to \eqref{eq1} in $(0,\infty)\times (\real^N\setminus B_r(0))$.
\end{lemma}

\begin{proof} The stated integrability of $ \Gamma_{p,q}$ follows from the property $\alpha/\beta>N$, see \eqref{expas2}, while a direct computation and the monotonicity of $\Gamma_{p,q}$ give the second assertion.
\end{proof}

\begin{lemma}\label{le.wp3}
Consider an initial condition $u_0$ satisfying \eqref{wp1} and let $u$ be the corresponding solution to \eqref{eq1}-\eqref{inco}.  Define
\begin{equation}\label{radius.initial}
R(u_0):=\inf\left\{ R>0: \ u_0(x)|x|^{\alpha/\beta}\leq\gamma \ a.\, e. \ \hbox{in} \ \{|x|\geq R\}\right\} \in [0,\infty].
\end{equation}
If $R(u_0)<\infty$, then
\begin{equation}
0 \leq u(t,x) \leq \Gamma_{p,q}(|x|-R(u_0)) \label{wp7}
\end{equation}
for any $t>0$ and $x\in\real^N$ with $|x|>R(u_0)$.
\end{lemma}

\begin{proof}
Clearly,
$$
u_0(x) \leq \gamma |x|^{-\alpha/\beta} = \Gamma_{p,q}(|x|-R(u_0))\,, \qquad x\in\real^N\setminus B_{R(u_0)}(0)\,.
$$
In addition, for all $x\in\real^N$ such that $|x| = R(u_0)$ and
$t>0$, we have $\Gamma_{p,q}(|x|-R(u_0)) = \infty > u(t,x)$. Thus,
$u(t,x) \leq \Gamma_{p,q}(|x|-R(u_0))$ on the parabolic boundary of
$(0,\infty)\times (\real^N\setminus B_{R(u_0)}(0))$, and the
comparison principle guarantees that $u(t,x)\leq
\Gamma_{p,q}(|x|-R(u_0))$ in $[0,\infty)\times \real^N\setminus
B_{R(u_0)}(0)$.
\end{proof}

We next prove an integral estimate on the tail behaviour of
solutions to \eqref{eq1}-\eqref{inco}.

\begin{lemma}\label{le.wp4}
Let $u_0$ be an initial condition satisfying \eqref{wp1} and denote the corresponding solution to \eqref{eq1}-\eqref{inco} by $u$. There is $C_0>0$ depending only on $N$, $p$, and $q$ such that, for $R>0$ and $t\ge 0$, there holds
\begin{equation}
\int_{\{|x|\ge R\}} u(t,x)\, dx \le C_0\ R^{(\beta N-\alpha)/\beta}\ \left( \sup_{|x|\ge R/2}{\left\{ u_0(x)\ |x|^{\alpha/\beta} \right\}} + t \ R^{-1/\beta} \right)\,. \label{wp5}
\end{equation}
\end{lemma}

\begin{proof}
We fix $\zeta\in C^\infty(\real^N)$ such that $0\le \zeta\le 1$ and
\begin{equation}
\zeta(x) = 0 \;\;\text{ if }\;\; |x|\le \frac{1}{2} \;\;\;\text{ and }\;\;\; \zeta(x) = 1 \;\;\text{ if }\;\; |x|\ge 1\,. \label{trunc}
\end{equation}
For $R>0$ and $x\in\real^N$, we define $\zeta_R(x):=\zeta(x/R)$. It follows from the weak formulation of \eqref{eq1} and Young's inequality that
\begin{align*}
& \frac{d}{dt} \int_{\real^N} \zeta_R(x)^{q/(q-p+1)}\ u(t,x)\, dx + \int_{\real^N} \zeta_R(x)^{q/(q-p+1)}\ |\nabla u(t,x)|^q\, dx \\
\le & \frac{q}{q-p+1}\ \int_{\real^N} \zeta_R(x)^{(p-1)/(q-p+1)}\ |\nabla u(t,x)|^{p-1}\ |\nabla\zeta_R(x)|\, dx \\
\le & \frac{p-1}{q-p+1}\ \int_{\real^N} \zeta_R(x)^{q/(q-p+1)}\ |\nabla u(t,x)|^q\, dx + \int_{\real^N} |\nabla\zeta_R(x)|^{q/(q-p+1)}\, dx \,,
\end{align*}
whence
\begin{equation}
\frac{d}{dt} \int_{\real^N} \zeta_R(x)^{q/(q-p+1)}\ u(t,x)\, dx \le C(\zeta)\ R^{(\beta N-\alpha-1)/\beta}\,. \label{fantasio}
\end{equation}
Owing to the properties \eqref{trunc} of $\zeta$, we find, after integrating with respect to time,
\begin{align*}
\int_{\{|x|\ge R\}} u(t,x)\, dx \le & \int_{\real^N} \zeta_R(x)^{q/(q-p+1)}\ u(t,x)\, dx \\
\le & \int_{\real^N} \zeta_R(x)^{q/(q-p+1)}\ u_0(x)\, dx + C(\zeta)\ t\ R^{(\beta N-\alpha-1)/\beta}\\
\le & \sup_{|x|\ge R/2}{\left\{ u_0(x)\ |x|^{\alpha/\beta} \right\}} \ \int_{\{ |x|\ge R/2\}} |x|^{-\alpha/\beta}\, dx + C(\zeta)\ t\ R^{(\beta N-\alpha-1)/\beta}\,,
\end{align*}
from which \eqref{wp5} follows.
\end{proof}

As a consequence of these integral tail estimates, we obtain some precise pointwise estimates for sufficiently rapidly decaying initial data.

\begin{lemma}\label{lem.point}
If $u_0$ satisfies \eqref{wp1} and \eqref{wp2} for some $\kappa>0$
and $u$ denotes the corresponding solution to the Cauchy problem
\eqref{eq1}-\eqref{inco}, then there exists $C>0$ depending on $N$,
$p$, and $q$ such that
\begin{equation}\label{est.point}
|x|^{\a/\b}u(t,x)\leq
C \left(\sup\limits_{|y|\geq|x|/4}\{u_0(y)|y|^{\a/\b}\}+t|x|^{-1/\b}\right)
\end{equation}
for any $x\in\real^N\setminus\{0\}$ and $t>0$.
\end{lemma}

\begin{proof}
\noindent \textbf{Step~1.} Let first $u_0$ be radially symmetric and non-increasing with respect to $|x|$. Then, by Lemma~\ref{lem.rs}, $u(t)$ has the same properties for any $t>0$, and for $x\in\real^N$, $x\neq0$ we deduce from Lemma~\ref{le.wp4} that
\begin{equation*}
\begin{split}
Cu(t,x)|x|^N&\leq\int_{\{|x|/2\leq|y|\leq|x|\}}u(t,y)\,dy\\&\leq
C_0\left(\frac{|x|}{2}\right)^{(N\b-\a)/\b}\left(\sup\limits_{|y|\geq|x|/4}\{u_0(y)|y|^{\a/\b}\}+t\left(\frac{2}{|x|}\right)^{1/\b}\right)\\
&\leq2^{(1+\a)/\b}C_0|x|^{(N\b-\a)/\b}\left(\sup\limits_{|y|\geq|x|/4}\{u_0(y)|y|^{\a/\b}\}+t|x|^{-1/\b}\right).
\end{split}
\end{equation*}
which gives \eqref{est.point} for this specific class of initial
data.

\medskip

\noindent \textbf{Step~2.} Fix $x_0\in\real^N\setminus\{0\}$. We
define
$$
\kappa_0 := \sup\limits_{|y|\geq|x_0|/4}\{u_0(y)|y|^{\a/\b}\}\le\kappa
$$
and take $R_0\in (0,|x_0|/4)$ such that $\kappa_0
R_0^{-\a/\b}\geq\|u_0\|_{\infty}$. We define
\begin{equation}
\tilde{u}_0(x):=\left\{\begin{array}{ll}2\kappa_0|x|^{-\a/\b}, \quad
|x|\geq R_0,\\2\kappa_0 R_0^{-\a/\b}, \quad |x|\leq
R_0.\end{array}\right.
\end{equation}
Then $\tilde{u}_0$ is a radially symmetric and non-increasing
function of $|x|$ and it satisfies \eqref{wp1} since $\a/\b>N$ as well as \eqref{wp2} with constant $2\kappa_0$. Moreover,
$u_0\leq\tilde{u}_0$ in $\real^N$, hence the comparison principle
guarantees that $u\leq\tilde{u}$ in $Q_{\infty}$, where $\tilde{u}$ is the solution to \eqref{eq1} with initial condition $\tilde{u}_0$. Applying Step~1 above to $\tilde{u}$ gives
\begin{align*}
|x_0|^{\a/\b} u(t,x_0) \le & |x_0|^{\a/\b} \tilde{u}(t,x_0) \le 2^{(1+\a)/\b} C_0 \left(\sup\limits_{|y|\geq|x_0|/4}\{\tilde{u}_0(y)|y|^{\a/\b}\}+t|x_0|^{-1/\b}\right) \\
\le & 2^{(1+\a)/\b} C_0 \left(  2\kappa_0 + t|x_0|^{-1/\b}\right)\,,
\end{align*}
and thus \eqref{est.point}.
\end{proof}

\section{Well-posedness with non-negative bounded measures as initial data}\label{sec3}

In this section, we prove Theorem~\ref{th.uniqfund}, together with
some preparatory results. We begin with the proof of the existence
statement which will be done, as usual, through an approximation
process.

\begin{proof}[Proof of Theorem~\ref{th.uniqfund}. Existence]
Let $u_0\in\cm_{b}^{+}(\real^N)$ and $(u_0^k)_{k\ge 1}$ be a sequence of functions in $C_0^\infty(\real^N)$ such that
\begin{equation}\label{aprox.cond1}
\|u_0^k\|_{1} = M_0 := \int_{\real^N} du_0\,,
\end{equation}
and
\begin{equation}\label{aprox.cond2}
\lim\limits_{k\to \infty} \int_{\real^N} u_0^k(x) \psi(x) \,dx =
\int_{\real^N} \psi(x)\,du_0(x)\ \ \hbox{for} \ \hbox{any} \
\psi\in BC(\real^N).
\end{equation}
Given $k\ge 1$, we denote the unique solution of \eqref{eq1} with initial condition $u_0^{k}$ by $u^{k}$. Owing to \eqref{aprox.cond1}, it follows from Proposition~\ref{pr.wp2a} that $(u^k)_k$ is bounded in $L^\infty(\tau,\infty;W^{1,\infty}(\real^N))$ for each $\tau>0$. Combining this property with Lemma~\ref{le.wp5} implies the time equicontinuity of the sequence $(u^k)_k$ in $(\tau,\infty)\times\real^N$ for all $\tau>0$. We then deduce from the Arzel\`a-Ascoli theorem that $(u^{k})_k$ is relatively compact in $C([\tau,T]\times K)$ for all compact subsets $K$ of $\real^N$ and $0<\tau< T$. There are thus a subsequence $(u^{k})$ (not relabeled) and a continuous function
$u\in C(Q_\infty)$ such that
\begin{equation}
u^{k}\longrightarrow u \;\;\text {in }\;\; C([\tau,T]\times K)
\;\;\text{ as }\;\; k\to \infty \label{gaston}
\end{equation}
for all compact subsets $K$ of $\real^N$ and $0<\tau<T$. Owing to the stability of viscosity solutions to \eqref{eq1} \cite[Theorem~6.1]{OS}, this convergence guarantees that $u$ is a viscosity solution to \eqref{eq1} in $Q_\infty$. In addition, since $u^k$ satisfies \eqref{wp101} and \eqref{wp102} with the constant $C_s(M_0)$, so does $u$. Consequently, $u(t)$ belongs to
$L^1(\real^N)$ and $W^{1,\infty}(\real^N)$ for all $t>0$.

In order to complete the proof of the existence part, it remains to
identify the initial condition taken by $u$. Consider $t\in (0,1)$, $\psi\in C_0^\infty(\real^N)$, and $k\ge 1$. Owing to \eqref{aprox.cond1}, we are in a position to apply Proposition~\ref{pr.wp2b}~(a) and conclude that
\begin{equation}\label{interm2}
\begin{split}
\left|\int_{\real^N} u^{k}(t,x) \psi(x) \,dx \right.&\left. - \int_{\real^N} u_0^{k}(x) \psi(x)\,dx \right|\\
&\leq C(M_0,1)\ \left( t^{1/p}\  \|\nabla\psi\|_{p/(2-p)} +
t^{(N+1)(q_*-q)\eta}\ \|\psi\|_{\infty} \right).
\end{split}
\end{equation}
Owing to \eqref{aprox.cond2} and \eqref{gaston}, we may let
$k\to\infty$ in \eqref{interm2} to get
\begin{equation*}
\begin{split}
\left|\int_{\real^N} u(t,x) \psi(x) \,dx \right.&\left. -\int_{\real^N} \psi(x) \,du_0(x)\right|\\
&\leq C\ \left( t^{1/p} \|\nabla\psi\|_{p/(2-p)} +
t^{(N+1)(q_*-q)\eta} \|\psi\|_{\infty} \right),
\end{split}
\end{equation*}
from which we readily deduce that
\begin{equation}\label{qqq}
\lim\limits_{t\to 0} \int_{\real^N} u(t,x) \psi(x) \,dx =
\int_{\real^N} \psi(x) \,du_0(x)
\end{equation}
for any $\psi\in C_0^\infty(\real^N)$. In fact, by a classical
density argument, \eqref{qqq} is valid for any continuous function
$\psi\in C_0(\real^N)$ which vanishes as $|x|\to\infty$. Let us now
show that \eqref{qqq} is satisfied for any function $\psi\in
BC(\real^N)$. To this end, let $\zeta\in C^\infty(\real^N)$ be such
that $0\le \zeta\le 1$ and
\begin{equation*}
\zeta(x) = 0 \;\;\text{ if }\;\; |x|\le \frac{1}{2} \;\;\;\text{ and }\;\;\; \zeta(x) = 1 \;\;\text{ if }\;\; |x|\ge 1\,,
\end{equation*}  and $\psi\in BC(\real^N)$. Then, for $R>0$, $\left( 1-\zeta_R^{q/(q-p+1)} \right)\psi$ belongs to $C_0(\real^N)$ and
\begin{align}
& \left| \int_{\real^N} u(t,x) \psi(x) \,dx-\int_{\real^N} \psi(x) \,du_0(x) \right| \nonumber \\
\le & \left| \int_{\real^N} u(t,x) \left( 1-\zeta_R(x)^{q/(q-p+1)} \right) \psi(x) \,dx - \int_{\real^N} \left( 1-\zeta_R(x)^{q/(q-p+1)} \right) \psi(x) \,du_0(x) \right| \nonumber \\
& + \int_{\real^N} u(t,x) \zeta_R(x)^{q/(q-p+1)} \psi(x)\, dx + \int_{\real^N} \zeta_R(x)^{q/(q-p+1)} \psi(x)\, du_0(x) \nonumber \\
\le & \left| \int_{\real^N} u(t,x) \left( 1-\zeta_R(x)^{q/(q-p+1)} \right) \psi(x) \,dx - \int_{\real^N} \left( 1-\zeta_R(x)^{q/(q-p+1)} \right) \psi(x) \,du_0(x) \right| \nonumber \\
& + \|\psi\|_\infty \left( \int_{\real^N} u(t,x) \zeta_R(x)^{q/(q-p+1)} \, dx + \int_{\real^N} \zeta_R(x)^{q/(q-p+1)} \, du_0(x) \right). \label{prunelle}
\end{align}
We now recall that it follows from \eqref{fantasio} that
$$
\int_{\real^N} u^k(t,x) \zeta_R(x)^{q/(q-p+1)} \, dx \le \int_{\real^N} u_0^k(x) \zeta_R(x)^{q/(q-p+1)} \, dx + C(\zeta) t R^{(\b N - \a -1)/\b}
$$
for $t\in (0,1)$ and $k\ge 1$. We then infer from \eqref{aprox.cond2}, \eqref{gaston}, and Fatou's lemma that
\begin{equation}
\int_{\real^N} u(t,x) \zeta_R(x)^{q/(q-p+1)} \, dx \le \int_{\real^N} \zeta_R(x)^{q/(q-p+1)} \, du_0(x) + C(\zeta) t R^{(\b N - \a -1)/\b} \label{lebrac}
\end{equation}
for $t\in (0,1)$.
We then infer from \eqref{qqq}, \eqref{prunelle}, and \eqref{lebrac} that
\begin{equation}
\limsup_{t\to 0} \left| \int_{\real^N} u(t,x) \psi(x) \,dx - \int_{\real^N} \psi(x) \,du_0(x) \right| \le 2\|\psi\|_\infty \int_{\real^N} \zeta_R(x)^{q/(q-p+1)} \, du_0(x). \label{contrex}
\end{equation}
Since $u_0$ is a bounded measure, we then let $R\to\infty$ in
\eqref{contrex} and use the properties of $\zeta$ to conclude that
the left-hand side of \eqref{contrex} vanishes. This ends the proof
of the existence result.
\end{proof}

\medskip

We next turn to the proof of the uniqueness part of Theorem~\ref{th.uniqfund} for which the following two preliminary results are needed. We will first need the following
inequality for vectors in $\real^N$.

\begin{lemma}\label{lemma.vecineq}
If $q\geq p/2$, then there exists $\vartheta=\vartheta(p,q)\in
(0,1]$ such that
\begin{equation}\label{vector.ineq}
(a-b)\cdot(|a|^{p-2}a-|b|^{p-2}b)\geq\vartheta\frac{\left||a|^{q-1}a-|b|^{q-1}b\right|^2}{|a|^{2q-p}+|b|^{2q-p}}\geq\vartheta\frac{\left(|a|^q-|b|^q\right)^2}{|a|^{2q-p}+|b|^{2q-p}},
\end{equation}
for all $(a,b)\in\real^N\times\real^N$.
\end{lemma}

When $q=1$ and $p\in(1,2]$, this lemma is proved in
\cite[Lemma~A.2]{BIV10}.

\begin{proof}
Consider $(a,b)\in\real^N\times \real^N$, $\vartheta\in (0,1]$, and
define
\begin{equation*}
\begin{split}
\Lambda(a,b)&:=(a-b)\cdot(|a|^{p-2}a-|b|^{p-2}b)\left(|a|^{2q-p}+|b|^{2q-p}\right)-\vartheta\left||a|^{q-1}a-|b|^{q-1}b\right|^2\\
&=\left[|a|^p+|b|^p-(|a|^{p-2}+|b|^{p-2})(a\cdot
b)\right]\left(|a|^{2q-p}+|b|^{2q-p}\right)\\&-\vartheta|a|^{2q}-\vartheta|b|^{2q}+2\vartheta|a|^{q-1}|b|^{q-1}(a\cdot
b)\\
&=\left(|a|^p+|b|^p\right)\left(|a|^{2q-p}+|b|^{2q-p}\right)-\vartheta\left(|a|^{2q}+|b|^{2q}\right)\\
&-\left[|a|^{2q-2}+|b|^{2q-2}+|a|^{p-2}|b|^{2q-p}+|a|^{2q-p}|b|^{p-2}-2\vartheta|a|^{q-1}|b|^{q-1}\right](a\cdot
b).
\end{split}
\end{equation*}
Since $\vartheta\in(0,1]$, we have
\begin{equation*}
\begin{split}
|a|^{2q-2}&+|b|^{2q-2}+|a|^{p-2}|b|^{2q-p}+|a|^{2q-p}|b|^{p-2}-2\vartheta|a|^{q-1}|b|^{q-1}\\
&\geq|a|^{2q-2}+|b|^{2q-2}-2|a|^{q-1}|b|^{q-1}=\left(|a|^{q-1}-|b|^{q-1}\right)^2\geq
0.
\end{split}
\end{equation*}
As $a\cdot b\leq|a||b|$, it follows from the previous inequalities
that
\begin{equation*}
\begin{split}
\Lambda(a,b)&\geq\left(|a|^p+|b|^p\right)\left(|a|^{2q-p}+|b|^{2q-p}\right)-\vartheta\left(|a|^{2q}+|b|^{2q}\right)\\
&-\left[|a|^{2q-2}+|b|^{2q-2}+|a|^{p-2}|b|^{2q-p}+|a|^{2q-p}|b|^{p-2}-2\vartheta|a|^{q-1}|b|^{q-1}\right]|a||b|\\
&\geq\left(|a|^p+|b|^p-|a|^{p-1}|b|-|a| |b|^{p-1}\right)\left(|a|^{2q-p}+|b|^{2q-p}\right)-\vartheta\left(|a|^q-|b|^q\right)^2\\
&\geq\left(|a|-|b|\right)\left(|a|^{p-1}-|b|^{p-1}\right)\left(|a|^{2q-p}+|b|^{2q-p}\right)-\vartheta\left(|a|^q-|b|^q\right)^2.
\end{split}
\end{equation*}
Since $q\geq p/2$, it follows from \cite[Lemma~1]{GP76} that there
is $C_1\geq 1$ depending only on $p$ and $q$ such that
\begin{equation*}
\frac{\left(|a|^q-|b|^q\right)^2}{\left(|a|^{p-1}-|b|^{p-1}\right)
\left(|a|-|b|\right)} \leq C_1 \max\left\{|a|,|b|\right\}^{2q-p}
\leq C_1\left(|a|^{2q-p}+|b|^{2q-p}\right).
\end{equation*}
Consequently, choosing $\vartheta=1/C_1$, we end up with
$\Lambda(a,b)\geq 0$, which implies the first inequality in
\eqref{vector.ineq}. The second inequality then follows easily from the triangular inequality.
\end{proof}

We next estimate the small time behavior of solutions to \eqref{eq1}.

\begin{lemma}\label{le.stb}
Consider $u_0\in\cm^{+}_{b}(\real^N)$ and let $u$ be a non-negative
solution to \eqref{eq1} with initial condition $u_0$. If there
exists a unique non-negative solution $v\in
C([0,\infty);\cm^{+}_{b}(\real^N)) \cap C(Q_\infty)$ to the
diffusion equation \eqref{wp103}-\eqref{wp104} in $Q_\infty$ with
initial condition $u_0$, then, for $t>0$ and $r\in [1,\infty]$,
\begin{equation}
\|u(t)\|_1 \le M_0 := \int_{\real^N} du_0(x)\,, \label{stb0}
\end{equation}
and
\begin{equation}
\| u(t) - v(t) \|_r \le C(M_0) \left( 1 + t^{(N+1)(q_*-q)q\eta/r(p-q)} \right)\ t^{[(N+1)(q_*-q)-N(r-1)]\eta/r}\,. \label{stb1}
\end{equation}
\end{lemma}

\begin{proof} For $\tau>0$, let $v^\tau$ be the solution to the diffusion equation \eqref{wp103} in $(\tau,\infty)\times\real^N$ with initial condition $v^\tau(\tau) =u(\tau)$.

We first prove \eqref{stb0}. By the comparison principle, $u\le v^\tau$ in $(\tau,\infty)\times\real^N$ while the $L^1$-accretivity of the $p$-Laplacian guarantees that $\|v^\tau(t)\|_1\le \|v^\tau(\tau)\|_1$ for $t>\tau$. Consequently, for $t>\tau$,
$$
\|u(t)\|_1 \le \|v^\tau(t)\|_1 \le \|v^\tau(\tau)\|_1=\int_{\real^N} u(\tau,x)\, dx \mathop{\longrightarrow}_{\tau\to 0} M_0\,,
$$
and thus \eqref{stb0}.

Next, since $u(\tau)\in L^1(\real^N)$ and $p>p_c$, it follows from the $L^1$-accretivity of the $p$-Laplacian that, for $t>\tau$,
\begin{equation*}
\|u(t)-v^\tau(t)\|_1 \leq \int_{\tau}^{t} \int_{\real^N} |\nabla
u(s,x)|^q \,dx\,ds.
\end{equation*}
Thanks to \eqref{stb0}, we may use \eqref{wp101} and  \eqref{wp100} to obtain
\begin{align*}
\|u(t)-v^\tau(t)\|_1 \leq & C\ \int_\tau^t \int_{\real^N} \left[ \left\| u\left( \frac{s+\tau}{2} \right) \right\|_\infty^{1/\alpha p} + (s-\tau)^{-1/p} \right]^q\ \left( u(s,x) \right)^{2q/p}\, ds \\
\le & C(M_0)\ \int_\tau^t \left[ (s-\tau)^{-qN\eta/\alpha p} + (s-\tau)^{-q/p} \right]\ \|u(s)\|_\infty^{(2q-p)/p}\ \|u(s)\|_1\, ds \\
\le & C(M_0)\ \int_\tau^t \left[ (s-\tau)^{-qN\eta/\alpha p} + (s-\tau)^{-q/p} \right]\ s^{-(2q-p)N\eta/p}\, ds \\
\le & C(M_0)\ \int_\tau^t \left[ (s-\tau)^{-N\eta/\alpha} + (s-\tau)^{-(q(N+1)-N)\eta} \right]\, ds \\
\le & C(M_0) \left[ t^{(N+1)(q_*-q)p\eta/(p-q)} + t^{(N+1)(q_*-q)\eta} \right]\,,
\end{align*}
hence
\begin{equation*}
\|u(t)-v^\tau(t)\|_1 \leq C(M_0) \left[ 1 + t^{(N+1)(q_*-q)q\eta/(p-q)} \right] t^{(N+1)(q_*-q)\eta} \,, \qquad t>\tau\,.
\end{equation*}
Now, since $v^\tau(t)$ converges towards $v(t)$ in $L^1(\real^N)$
for all $t>0$, we conclude that
\begin{equation}
\|u(t)-v(t)\|_1 \leq C(M_0) \left[ 1 + t^{(N+1)(q_*-q)q\eta/(p-q)} \right] t^{(N+1)(q_*-q)\eta}\,, \qquad t>0 \,. \label{intermL1}
\end{equation}
Also, by \eqref{wp101}, \eqref{wp300}, and \eqref{stb0},
\begin{equation}\label{intermLinf}
\|u(t) - v(t)\|_{\infty} \leq \|u(t)\|_{\infty} + \|v(t)\|_{\infty} \leq C(M_0)\ t^{-N\eta}\,, \quad t>0.
\end{equation}
We then infer from \eqref{intermL1}, \eqref{intermLinf}, and
H\"older's inequality that \eqref{stb1} is true.
\end{proof}

\begin{proof}[Proof of Theorem~\ref{th.uniqfund}. Uniqueness]
Let $u_1$ and $u_2$ be two non-negative solutions to \eqref{eq1} with initial condition $u_0\in\cm^{+}_{b}(\real^N)$ and define
$$
M_0 := \int_{\real^N} du_0(x)\,,
$$
and $w:=u_1-u_2$. Then, $w$ solves
\begin{equation}
\partial_{t}w-(\Delta_{p}u_1-\Delta_{p}u_2)+|\nabla u_1|^q-|\nabla
u_2|^q=0 \quad \hbox{in} \ Q_\infty.
\end{equation}
Consider $r>0$ to be specified later and $T>0$. For $t\in (0,T)$, we calculate
\begin{equation*}
\begin{split}
\frac{1}{r+1}\ \frac{d}{dt} \|w\|_{r+1}^{r+1}
&=-\int_{\real^N}r|w|^{r-1}\nabla w\cdot\left(|\nabla
u_1|^{p-2}\nabla u_1-|\nabla u_2|^{p-2}\nabla u_2\right)\,dx\\
&-\int_{\real^N}|w|^{r-1}w\left(|\nabla u_1|^q-|\nabla
u_2|^q\right)\,dx.
\end{split}
\end{equation*}
Lemma~\ref{lemma.vecineq} then gives, with the help of Young's
inequality,
\begin{align*}
& \frac{1}{r+1}\ \frac{d}{dt} \|w\|_{r+1}^{r+1} \\
\leq & -r\vartheta \int_{\real^N}|w|^{r-1} \frac{(|\nabla
u_1|^q-|\nabla u_2|^q)^2}{1+|\nabla u_1|^{2q-p}+|\nabla
u_2|^{2q-p}}\,dx\\
& + \int_{\real^N}|w|^{(r+1)/2}\frac{|w|^{(r-1)/2}(|\nabla
u_1|^q-|\nabla u_2|^q)}{\sqrt{1+|\nabla u_1|^{2q-p}+|\nabla
u_2|^{2q-p}}} \sqrt{1+|\nabla u_1|^{2q-p}+|\nabla u_2|^{2q-p}}\,dx\\
\leq & C\ \int_{\real^N}|w|^{r+1} \left(1+ |\nabla u_1|^{2q-p}+|\nabla u_2|^{2q-p}\right)\,dx\\
\leq & C\left(1+\|\nabla u_1\|_{\infty}^{2q-p}+\|\nabla
u_2\|_{\infty}^{2q-p}\right)\|w\|_{r+1}^{r+1}.
\end{align*}
Owing to \eqref{stb0}, we are in a position to use the gradient estimate \eqref{wp102} and we further obtain
\begin{equation*}
\frac{1}{r+1}\ \frac{d}{dt} \|w(t)\|_{r+1}^{r+1}\leq
C(M_0,T) \left( 1 + t^{-(N+1)\eta(2q-p)} \right) \|w(t)\|_{r+1}^{r+1}.
\end{equation*}
Observing that
$$
1-(N+1)\eta(2q-p)=2(N+1)(q_*-q)\eta>0,
$$
we may  integrate the above differential inequality over $(s,t)$, $0<s<t<T$, to obtain
\begin{equation}\label{interm5}
\|w(t)\|_{r+1}^{r+1} \leq \|w(s)\|_{r+1}^{r+1}\ \exp\left\{(r+1) C(M_0,T) \left( t^{2(N+1)(q_*-q)\eta} + t \right)\right\}.
\end{equation}
We now choose $r\in(0,(N+1)(q_*-q)/N)$ and realize that
\eqref{stb1} guarantees that (keeping the notation of Lemma~\ref{le.stb})
\begin{align*}
\|w(s)\|_{r+1}^{r+1} \le & \|u_1(s)-v(s)\|_{r+1}^{r+1} + \|v(s)-u_2(s)\|_{r+1}^{r+1} \\
\leq & C(M_0,T)\ s^{((N+1)(q_*-q)-Nr)\eta}  \mathop{\longrightarrow}_{s\to 0} 0.
\end{align*}
Consequently, letting $s\to 0$ in  \eqref{interm5} leads us to $\|w(t)\|_{r+1}^{r+1} \leq 0$ for all $t\in (0,T)$, hence $u_1\equiv u_2$ in $(0,T)$. As $T$ was arbitrary, the proof is complete.
\end{proof}

Since initial data of the form $M\delta_0$, where $\delta_0$ denotes the Dirac mass at $x=0$, play an essential role in the sequel, we rephrase Theorem~\ref{th.uniqfund} in this particular setting.

\begin{corollary}\label{cor.fund}
For any $M>0$, there exists a unique solution $u_{M}$ to \eqref{eq1}
with initial condition $M\delta_0$. In the sequel, $u_M$ will be
refered to as \emph{the fundamental solution to \eqref{eq1} of mass
$M$}. Moreover, it satisfies the estimates \eqref{wp101} and
\eqref{wp102} with $C_s(M)$ instead of $C_s(M_0)$.
\end{corollary}

\begin{proof}
The existence and uniqueness of a solution for the $p$-Laplacian
equation \eqref{wp103} with initial condition $u_0=M\delta_0$ are proved in \cite[Theorem 4.1]{CQW07}. Thus, applying Theorem~\ref{th.uniqfund} with $u_0=M\delta_0$, we get the claimed result.
\end{proof}


\section{Very singular solutions}\label{sec4}

As specified in the Introduction, we will study in detail the very
singular solutions of \eqref{eq1}. More precisely, in this section
we show that there exists in fact a unique very singular solution to \eqref{eq1}. This is done by constructing a minimal and a maximal very singular solution and identifying them afterwards. We begin with the precise definition.

\begin{definition}\label{def.VSS}
A very singular solution to \eqref{eq1} is a viscosity solution $u$ to \eqref{eq1} in $Q_\infty$ in the sense of Definition~\ref{def.visc} satisfying
\begin{equation}
u(t)\in L^1(\real^N)\cap W^{1,\infty}(\real^N) \label{regvss}
\end{equation}
for all $t>0$ as well as
\begin{equation}\label{VSS1}
\lim\limits_{s\to 0}\int_{\{|x|\leq r\}}u(s,x)\,dx = \infty, \quad
r\in(0,\infty),
\end{equation}
and
\begin{equation}\label{VSS2}
\lim\limits_{s\to 0}\int_{\{|x|\geq r\}}u(s,x)\,dx=0, \quad
r\in(0,\infty).
\end{equation}
A very singular subsolution (resp. supersolution) to \eqref{eq1} is
a viscosity subsolution (resp. supersolution) to \eqref{eq1} in $Q_\infty$ in the sense of
Definition~\ref{def.visc}, which satisfies \eqref{regvss},
\eqref{VSS1} and \eqref{VSS2}.
\end{definition}

We already know that the class of very singular solutions to
\eqref{eq1} for $p\in(p_c,2)$ and $q\in(p/2,q_*)$ is non-empty as a
consequence of the following result \cite{IL2}.

\begin{theorem}\label{th.VSSunique}
There exists a unique radially symmetric, self-similar very singular solution $U$ to \eqref{eq1}, having the form
\begin{equation}\label{selfVSS}
U(t,x)=t^{-\alpha}f_U(|x|t^{-\beta}), \quad
(t,x)\in Q_\infty.
\end{equation}
The profile $f_U$ is a solution to the differential equation
\begin{equation}\label{ODE1}
(|f_U'|^{p-2}f_U')'(r)+\frac{N-1}{r}(|f_U'|^{p-2}f_U')(r)+\a f_U(r)+\b
rf_U'(r)-|f_U'(r)|^q=0\,, \quad r> 0,
\end{equation}
satisfying $f_U'(0)=0$ and there is an explicit positive constant
$\omega^*$ such that
$$
\lim\limits_{r\to\infty} r^{p/(2-p)}\ f_U(r) = \omega^*.
$$
\end{theorem}

This important result is very useful in the sequel in order to
identify very singular solutions when we are able to show that they
are radially symmetric and in self-similar form.

\subsection{Some properties of very singular subsolutions and solutions}

From Definition~\ref{def.VSS}, one expects the initial trace of
a very singular solution to \eqref{eq1} to vanish outside the
origin. This is made rigorous in the next result.

\begin{proposition}\label{prop.zero}
Let $u$ be a very singular subsolution to \eqref{eq1} and $K$ be a compact subset of $\real^N\setminus\{0\}$. Then
$$
\lim\limits_{t\to 0} \sup_{x\in K}\{u(t,x)\} = 0.
$$
\end{proposition}

\begin{proof}
Fix $\tau>0$ and let $v_{\tau}$ be the solution to the diffusion equation \eqref{wp103} in $(\tau,\infty)\times\real^N$ with initial condition $v_\tau(\tau)=u(\tau)$. According to \cite[Theorem III.6.2]{DBH90}, $v_\tau$ satisfies the following pointwise estimate: there exists a constant $C>0$ depending only on $N$ and $p$ such that, for any $x_0\in\real^N$, $R>0$,
and $t>\tau$,
\begin{equation}\label{DiBH}
\sup_{x\in B_R(x_0)}\{v_{\tau}(t,x)\} \leq C\ (t-\tau)^{-N\eta} \left(\int_{B_{2R}(x_0)} v_{\tau}(\tau,x)\,dx\right)^{p\eta} + C \left(\frac{t-\tau}{R^p}\right)^{1/(2-p)}.
\end{equation}
Since $u$ is a subsolution to the diffusion equation \eqref{wp103} in $(\tau,\infty)\times\real^N$ with $u(\tau)=v_\tau(\tau)$, the comparison principle gives $u\le v_\tau$ in $(\tau,\infty)\times\real^N$. Plugging these information in \eqref{DiBH}, we are led to
\begin{equation}\label{DiBH2}
\sup_{x\in B_R(x_0)}\{u(t,x)\} \leq C \ (t-\tau)^{-N\eta} \left(\int_{B_{2R}(x_0)} u(\tau,x) \,dx \right)^{p\eta}+ C \left(\frac{t-\tau}{R^p}\right)^{1/(2-p)}
\end{equation}
for any $t>\tau>0$. Now, assume further that $x_0\ne 0$ and $|x_0|>2R$. Then $0\not\in B_{2R}(x_0)$ and we may let $\tau\to 0$ in \eqref{DiBH2} and use \eqref{VSS2} to obtain
\begin{equation}\label{DiBH3}
\sup_{x\in B_R(x_0)}\{u(t,x)\}\leq C \left(\frac{t}{R^p}\right)^{1/(2-p)}\,, \qquad t>0\,.
\end{equation}
Therefore, if $x_0\ne 0$ and $|x_0|>2R$,
$$
\lim\limits_{t\to 0}\ \sup_{x\in B_R(x_0)}\{u(t,x)\} = 0,
$$
and this property entails Proposition~\ref{prop.zero} by a covering argument.
\end{proof}

In particular, Proposition~\ref{prop.zero} implies that $u(0,x)=0$,
for any very singular subsolution $u$ and any $x\neq0$. This is
useful to prove some comparison results.

\begin{proposition}\label{prop.compFG}
Let $u$ be a very singular subsolution to \eqref{eq1}. Then
\begin{equation}\label{comp.VSS}
0\leq u(t,x)\leq\Gamma_{p,q}(|x|), \quad (t,x)\in Q_\infty.
\end{equation}
\end{proposition}

\begin{proof}
We adapt the proof of \cite[Lemma~3.4]{BKL04}. At a formal level,
the result follows from Lemma~\ref{le.wp3} since we can view a very singular solution as having an initial condition satisfiying $R(u_0)=0$. More precisely, let $r>0$ and define $D_{r}:=\{x\in\real^N: \ |x|>r\}$. By Lemma~\ref{le.wp6}, $S:(t,x)\longmapsto \Gamma_{p,q}(|x|-r)$ is a supersolution to \eqref{eq1} in $(0,\infty)\times D_r$ with $u(t,x)<\infty=S(t,x)$ if $(t,x)\in (0,\infty)\times \partial D_r$ and $u(0,x)=0\le S(x)$ for $x\in D_r$ by Proposition~\ref{prop.zero}. Since $u$ is a subsolution to \eqref{eq1} in $Q_\infty$ and thus also in $(0,\infty)\times D_r$, the comparison principle gives
$u(t,x)\leq\Gamma_{p,q}(|x|-r)$ for any $(t,x)\in(0,\infty)\times
D_r$. Fix now $x_0\in\real^N$, $x_0\ne 0$. Then $x_0\in D_r$ for any $r\in(0,|x_0|)$, hence $u(t,x_0)\leq\Gamma_{p,q}(|x_0|-r)$, for any $t>0$ and $r\in(0,|x_0|)$. The conclusion follows by letting $r\to 0$ in the previous inequality.
\end{proof}

We next prove that very singular subsolutions also enjoy the temporal decay estimates \eqref{wp3}.

\begin{proposition}\label{prop.decayVSS}
If $u$ is a very singular subsolution to \eqref{eq1} in $Q_\infty$, the following estimates hold:
\begin{equation}\label{decayVSS}
t^{\alpha-N\beta}\|u(t)\|_{1} + t^{\alpha}\|u(t)\|_{\infty} \leq K_\gamma, \quad t>0,
\end{equation}
where $\gamma$ and $K_\gamma$ are defined in \eqref{exp.FG} and Proposition~\ref{pr.wp2}, respectively. In addition, if $u$ is a very singular solution to \eqref{eq1} in $Q_\infty$,
\begin{equation}
t^{\alpha+\beta} \|\nabla u(t)\|_{\infty} \leq K_\gamma, \quad t>0. \label{gradestVSS}
\end{equation}
\end{proposition}

\begin{proof}
At a formal level, since $u$ is a very singular subsolution, its
initial condition is somehow concentrated at $x=0$. It thus
``vanishes'' outside the origin and the conditions on the initial
data in Proposition~\ref{pr.wp2} are fulfilled. As more regularity
on the initial condition is needed to apply this result, we provide
a rigorous proof now. Consider $\tau>0$. According to
\eqref{regvss} and \eqref{comp.VSS}, $u(\tau)$ satisfies \eqref{wp1} and \eqref{wp2} with $\kappa=\gamma$ and we infer from Proposition~\ref{pr.wp2} that the solution $u^\tau$ to \eqref{eq1} in $(\tau,\infty)\times\real^N$ with initial condition $u^\tau(\tau)=u(\tau)$ satisfies
\begin{align*}
(t-\tau)^{\alpha-N\beta} \|u^\tau(t)\|_{1} + (t-\tau)^{\alpha} \|u^\tau(t)\|_{\infty} \leq & K_\gamma, \\
(t-\tau)^{\alpha+\beta} \|\nabla u^\tau(t)\|_{\infty} \leq & K_\gamma,
\end{align*}
for $t>\tau$. Now, if $u$ is a very singular subsolution to \eqref{eq1}, the comparison principle gives $u\le u^\tau$ in $(\tau,\infty)\times\real^N$ and \eqref{decayVSS} follows at once from the previous estimate after letting $\tau\to 0$. Next, if $u$ is a very singular subsolution to \eqref{eq1}, we obviously have $u^\tau=u$ and thus \eqref{gradestVSS}.
\end{proof}

Finally, the last preliminary result concerns some local estimates
on small balls for very singular subsolutions. It is similar to \cite[Lemma 3.6]{BKL04} for $p=2$,  and its proof adapts an argument from \cite[p. 186-187]{CLS}.

\begin{proposition}\label{prop.local}
For $y\in\real^N$ and $\varrho>0$, let $\sigma_{y,\varrho}$ be
the solution to
\begin{equation}
-\Delta_{p}\sigma_{y,\varrho}=1 \quad \hbox{in}\
\ B_\varrho(y),\qquad \sigma_{y,\varrho}=0 \quad \hbox{on} \ \
\partial B_\varrho(y).\label{sigma}
\end{equation}
For every $\lambda\in (0,\infty)$, there exists $A_{\lambda,\varrho}>0$ depending only on $N$, $p$, $\varrho$, and $\lambda$ such that, if $u$ is a very singular subsolution
to \eqref{eq1}, $y\in\real^N\setminus\{0\}$, and $0<\varrho<|y|$, we have
\begin{equation}\label{locest.VSS}
u(t,x) \leq \lambda\
e^{A_{\lambda,\varrho}t}\ \exp\left(\frac{1}{\sigma_{y,\varrho}(x)}\right), \quad (t,x)\in(0,\infty)\times B_\varrho(y).
\end{equation}
\end{proposition}

\begin{proof}
We fix $y\in\real^N$, $\varrho\in (0,|y|)$, $\lambda>0$, and define
$$
w(t,x):=\lambda\ e^{At}\ \exp\left(\frac{1}{\sigma(x)}\right), \qquad (t,x)\in (0,\infty)\times B_\varrho(y),
$$
where $\sigma=\sigma_{y,\varrho}$, the dependence on $y$ and
$\varrho$ being omitted for simplicity. We wish to choose $A>0$
such that
\begin{equation}\label{interm6}
\partial_t w - \Delta_{p}w + |\nabla w|^q\geq 0 \quad \hbox{in} \ (0,\infty)\times B_\varrho(y).
\end{equation}
To this end, we calculate:
$$
\partial_t w(t,x)=A\ w(t,x), \quad \nabla
w(t,x)=-\frac{w(t,x)}{\sigma(x)^2}\ \nabla\sigma(x),
$$
hence
$$
|\nabla w(t,x)|^q = \frac{|\nabla\sigma(x)|^q}{\sigma(x)^{2q}}\ w(t,x)^q
$$
and
$$
|\nabla w(t,x)|^{p-2}\nabla w(t,x) = - \frac{|\nabla\sigma(x)|^{p-2}}{\sigma(x)^{2(p-1)}}\ w(t,x)^{p-1}\ \nabla\sigma(x).
$$
It follows from \eqref{sigma} that
\begin{equation*}
\begin{split}
\Delta_{p} w(t,x) = & 2(p-1)\ \frac{|\nabla\sigma(x)|^p}{\sigma(x)^{2p-1}}\ w(t,x)^{p-1} + (p-1)\ \frac{|\nabla\sigma(x)|^p}{\sigma(x)^{2p}}\ w(t,x)^{p-1} \\
& - \frac{w(t,x)^{p-1}}{\sigma(x)^{2(p-1)}}\ \Delta_p\sigma(x)\\
= & \frac{w(t,x)^{p-1}}{\sigma(x)^{2p}}\ \left[ \sigma(x)^{2} + (p-1)\ (1+2\sigma(x))\ |\nabla\sigma(x)|^{p} \right].
\end{split}
\end{equation*}
Gathering all the previous calculations, we obtain
\begin{align*}
\partial_t w - \Delta_{p} w + |\nabla
w|^q = & w^{p-1} \left\{ A\ w^{2-p} - \frac{\sigma^2 + (p-1)(1+2\sigma) |\nabla\sigma|^p}{\sigma^{2p}} + \frac{|\nabla\sigma|^{q}}{\sigma^{2q}}\ w^{q-p+1} \right\} \\
\ge & w^{p-1} \left\{ \lambda^{2-p} A\ \exp{\left\{ \frac{2-p}{\sigma}
\right\}} - \frac{\left\| \sigma^2 + (1+2\sigma) |\nabla\sigma|^p
\right\|_{L^{\infty}(B_{\varrho}(y))}}{\sigma^{2p}} \right\}.
\end{align*}
Setting $\mu_p := \inf_{r>0}{\left\{ e^r\ r^{-2p} \right\}}>0$, we end up with
$$
\partial_t w - \Delta_{p} w + |\nabla
w|^q \ge \frac{w^{p-1}}{\sigma^{2p}} \left\{ \lambda^{2-p} (2-p)^{2p}
\mu_p A - \left\| \sigma^2 + (1+2\sigma) |\nabla\sigma|^p
\right\|_{L^{\infty}(B_{\varrho}(y))} \right\}.
$$
Since $\sigma(x) = \varrho^p\ \sigma_{0,1}((x-y)/\varrho)$ for $x\in
B_\varrho(y)$, we conclude that \eqref{interm6} holds true for a
sufficiently large constant $A_{\lambda,\varrho}>0$ which depends
only on $N$, $p$, $\lambda$, and $\varrho$.

With this choice, $w$ is a supersolution to \eqref{eq1} in
$(0,\infty)\times B_\varrho(y)$ which satisfies additionally
$w(0,x)\ge 0 = u(0,x)$ for $x\in B_\varrho(y)$ by
Proposition~\ref{prop.zero} and $w(t,x)=\infty>u(t,x)$ for $(t,x)\in
(0,\infty)\times\partial B_\varrho(y)$ by \eqref{sigma}. The
estimate \eqref{locest.VSS} then follows by the comparison
principle.
\end{proof}

\subsection{The minimal very singular solution}\label{sec4min}

In this section we will construct a special very singular solution
and prove that it is minimal among all the very singular solutions
and has a self-similar form. As a consequence, it will coincide with
the unique radially symmetric self-similar very singular solution
obtained in \cite{IL2}, see Theorem~\ref{th.VSSunique}. Recalling
the notation $u_M$ for the fundamental solution to \eqref{eq1} with
mass $M>0$, we begin with the following preliminary result.

\begin{lemma}\label{lemma.fs1}
Let $u$ be a very singular supersolution to \eqref{eq1} and assume further that
$$
u\in C(Q_{\infty}) \ {\rm and} \ u(t,x)\leq \Gamma_{p,q}(|x|),
\qquad (t,x)\in Q_\infty\,.
$$
Then, for any $M>0$, we have $u_M\leq u$ in $Q_\infty$.
\end{lemma}

\begin{proof}
Fix $M>0$. We borrow ideas from the proofs of \cite[Lemma 3.7]{BKL04} and Theorem~\ref{th.uniqfund} above. As $u$ is
a very singular supersolution to \eqref{eq1}, we have
$\|u(t)\|_{1}\longrightarrow\infty$ as $t\to 0$ and, for each $k\ge 1$, there exists a non-negative function $u_{0,k}\in L^1(\real^N)\cap W^{1,\infty}(\real^N)$ such that
\begin{equation}\label{interm7}
\|u_{0,k}\|_1=M, \quad 0\le u_{0,k}(x)\leq u(1/k,x) \le \Gamma_{p,q}(|x|), \ \hbox{for} \ \hbox{any} \ x\in\real^N.
\end{equation}
Denoting the solution to \eqref{eq1} with initial condition $u_{0,k}$ by $u_k$, we argue as in the proof of the existence part of Theorem~\ref{th.uniqfund} to
find a non-negative function $\tilde{u}\in C(Q_\infty)$ and a subsequence of $(u_k)_k$ (not relabeled) with the following properties:
\begin{equation}
\begin{minipage}{12cm}
$\tilde{u}$ is a solution to \eqref{eq1} in $Q_\infty$ and satisfies the estimates \eqref{wp101}-\eqref{wp102} with $C_s(M)$ and \eqref{wp3} with $\kappa=\gamma$.
\end{minipage}
\label{pim}
\end{equation}
and
\begin{equation}
u_k\longrightarrow \tilde{u} \;\;\text{ in }\;\; C([\tau,T]\times K)
\label{pam}
\end{equation}
for all compact subsets $K$ of $\real^N$ and $\tau<t<T$.

It remains to identify the initial condition taken by $\tilde{u}$. On the one hand, since $u$ is a supersolution to \eqref{eq1}, it readily follows from \eqref{interm7} that
$$
u_k(t,x) \le u\left( t + \frac{1}{k} , x \right) \le \Gamma_{p,q}(|x|)\,, \qquad (t,x)\in Q_\infty\,,
$$
whence, owing to \eqref{pam} and the continuity of $u$ in
$Q_\infty$,
\begin{equation}
\tilde{u}(t,x) \le u(t,x) \le \Gamma_{p,q}(|x|)\,, \qquad (t,x)\in Q_\infty\,. \label{poum}
\end{equation}
On the other hand, consider $\psi\in C_0^{\infty}(\real^N)$ and $t\in(0,1)$. Owing to \eqref{interm7}, we may use Proposition~\ref{pr.wp2b}~(a) and deduce that, for all $k\ge 1$,
\begin{equation}
\begin{split}
\left| \int_{\real^N} \left( u_k(t,x)- u_{0,k}(x) \right) \psi(x)\,dx \right| \le C(M,1)\ & \left[ \|\psi\|_\infty\ t^{(N+1)(q_*-q)\eta} \right. \\ & + \left. \|\nabla\psi\|_{p/(2-p)}\ t^{1/p} \right].
\end{split} \label{b00}
\end{equation}
It also follows from \eqref{interm7} that, for $r>0$ and $k\ge 1$,
\begin{align*}
&\left|\int_{\real^N} u_{0,k}(x) \psi(x) \,dx - M\ \psi(0)\right| = \left|\int_{\real^N} u_{0,k}(x) (\psi(x)-\psi(0)) \,dx \right|\\
&\leq 2\|\psi\|_{\infty}\ \int_{\{|x|\geq r\}} u(1/k,x) \,dx + \left( \int_{\{|x|\leq
r\}} u_{0,k}(x) \,dx\right)\ \sup\limits_{|x|\leq
r}{\left\{ |\psi(x)-\psi(0)| \right\}} \\
& \leq 2\|\psi\|_{\infty}\int_{\{|x|\geq
r\}} u(1/k,x) \,dx + M \sup\limits_{|x|\leq
r}{\left\{ |\psi(x)-\psi(0)| \right\}}.
\end{align*}
Combining \eqref{b00} and the above estimate, we obtain, for $k\ge 1$ and $r>0$,
\begin{align*}
& \left| \int_{\real^N} \tilde{u}(t,x)\ \psi(x)\, dx - M\ \psi(0) \right| \\
\le & \left| \int_{\real^N} \left( \tilde{u}(t,x) - u_k(t,x) \right) \psi(x)\, dx \right| + \left|\int_{\real^N} \left( u_k(t,x) - u_{0,k}(x) \right) \psi(x)\,dx \right| \\
& + \left| \int_{\real^N} u_{0,k}(x)\ \psi(x) \,dx - M \psi(0) \right|\\
\le & \int_{\real^N} \left| \tilde{u}(t,x) - u_k(t,x) \right| |\psi(x)|\, dx \\
& + C(M,1)\ \left[ \|\psi\|_\infty\ t^{(N+1)(q_*-q)\eta} + \|\nabla\psi\|_{p/(2-p)}\ t^{1/p} \right] \\
& + 2\|\psi\|_{\infty}\int_{\{|x|\geq
r\}} u(1/k,x) \,dx + M \sup\limits_{|x|\leq
r}{\left\{ |\psi(x)-\psi(0)| \right\}}.
\end{align*}
Since $t>0$, $r>0$, and $\psi$ is compactly supported, we first let $k\to\infty$ in the above inequality and use \eqref{VSS2} and \eqref{pam} to conclude that
\begin{align*}
\left| \int_{\real^N} \tilde{u}(t,x)\ \psi(x)\, dx - M\ \psi(0) \right| \le & C(M,1)\ \left[ \|\psi\|_\infty\ t^{(N+1)(q_*-q)\eta} + \|\nabla\psi\|_{p/(2-p)}\ t^{1/p} \right] \\
& + M \sup\limits_{|x|\leq
r}{\left\{ |\psi(x)-\psi(0)| \right\}}.
\end{align*}
We then let $t\to 0$ and $r\to 0$ and end up with
\begin{equation}\label{interm14}
\lim\limits_{t\to 0} \int_{\real^N} \tilde{u}(t,x)\ \psi(x) \,dx = M\ \psi(0)
\end{equation}
for any $\psi\in C_0^{\infty}(\real^N)$. By a standard density
argument, we extend \eqref{interm14} to test functions $\psi\in
C_0(\real^N)$. In order to extend \eqref{interm14} to test functions in $BC(\real^N)$, we proceed as in the proof of the existence part of Theorem~\ref{th.uniqfund} with the difference that the control for large $x$ is here provided by $\Gamma_{p,q}$ thanks to the upper bound \eqref{poum} and Lemma~\ref{le.wp6}. The uniqueness statement of Theorem~\ref{th.uniqfund} then implies that $\tilde{u}=u_M$. Recalling \eqref{poum} completes the proof.
\end{proof}

The next result shows more properties of the fundamental solutions $u_M$.

\begin{lemma}\label{lemma.fs2}
\begin{itemize}
\item[(a)] For each $M>0$ and $t>0$, $u_{M}(t)$ is a radially symmetric function, and $u_{M_1}(t)\leq u_{M_2}(t)$ if $0<M_1\leq M_2<\infty$.

\item[(b)] For each $M>0$, the function $u_M$ satisfies
\begin{equation}\label{FS.bound1}
0\leq u_{M}(t,x)\leq \Gamma_{p,q}(|x|), \quad (t,x)\in Q_\infty
\end{equation}
as well as the estimates \eqref{wp101}-\eqref{wp102} with $C_s(M)$ and \eqref{wp3} with $\kappa=\gamma$.

\item[(c)] For each $M>0$ and any $r>0$, there exist a constant $C(r)$ depending only on $r$, $p$, $q$, and $N$ such that
\begin{equation}\label{FS.bound2}
\int_{\{|x|\geq r\}} u_{M}(t,x)\,dx\leq C(r)\ t^{1/p}\,, \qquad t\in (0,1).
\end{equation}
\end{itemize}
\end{lemma}

\begin{proof}
The proof of part~(a) is identical to the proof of
\cite[Lemma~3.3]{BL01} to which we refer. Next, it is easy to see that Proposition~\ref{prop.zero} is also valid for the fundamental solutions $u_M$ and the estimate~\eqref{FS.bound1}
can be proved as Proposition~\ref{prop.compFG}. We then infer from \eqref{stb0} and \eqref{FS.bound1} that Propositions~\ref{pr.wp2a} and~\ref{pr.wp2} can be applied to $(t,x)\longmapsto u_M(t+\tau,x)$ for any arbitrary small $\tau$ from which the validity of \eqref{wp101}-\eqref{wp102} with $C_s(M)$ and \eqref{wp3} with $\kappa=\gamma$ follows after passing to the limit $\tau\to 0$. Finally, let $\tau>0$, $r>0$, and two non-negative functions $\xi\in C_0^\infty(\real^N)$ and $\zeta\in C^\infty(\real^N)$ such that
$$
0\le \xi\le 1\,, \qquad \xi(x)=1 \;\;\text{ if }\;\; |x|<1 \;\;\text{ and }\;\; \xi(x)=0 \;\;\text{ if }\;\; |x|>2
$$
and
$$
0\le \zeta\le 1\,, \qquad \zeta(x)=0 \;\;\text{ if }\;\; |x|<r/2 \;\;\text{ and }\;\;\zeta(x)=1 \;\;\text{ if }\;\; |x|>r\,.
$$
For $R>0$ and $x\in\real^N$, we set $\xi_R(x)=\xi(x/R)$. Since $u(\tau)$ satisfies \eqref{wp2} with $\kappa=\gamma$ by \eqref{FS.bound1} and $\xi_R \zeta\in C_0^\infty(\real^N)$ vanishes in $B_{r/2}(0)$, it follows from \eqref{wp201} that, for $t>\tau$ and $R>0$,
\begin{align*}
\int_{\{r<|x|<R\}} u_M(t,x)\, dx \le & \int_{\real^N} u_M(t,x)\ (\xi_R \zeta)(x)\, dx \\
\le & \int_{\real^N} u_M(\tau)\ (\xi_R \zeta)(x)\, dx + C(\gamma,r/2)\ \|\nabla(\xi_R \zeta)\|_{p/(2-p)}\ (t-\tau)^{1/p}\,.
\end{align*}
Letting $\tau\to 0$, we find, since $\xi_R \zeta$ vanishes in a neighborhood of $x=0$,
$$
\int_{\{r<|x|<R\}} u_M(t,x)\, dx \le C(\gamma,r/2)\ \|\nabla(\xi_R \zeta)\|_{p/(2-p)}\ t^{1/p}\,.
$$
Combining \eqref{FS.bound1} with the previous inequality, we obtain
\begin{align*}
\int_{\{|x|>r\}} u_M(t,x)\, dx \le & \int_{\{r<|x|<R\}} u_M(t,x)\, dx + \int_{\{|x|>R\}} \Gamma_{p,q}(|x|)\, dx \\
\le & C(\gamma,r/2)\ \|\nabla(\xi_R \zeta)\|_{p/(2-p)}\ t^{1/p} + \int_{\{|x|>R\}} \Gamma_{p,q}(|x|)\, dx \\
\le & C(r)\ \left( \|\nabla(\xi_R \zeta)\|_{p/(2-p)}\ t^{1/p} + R^{-(\alpha-N\beta)/\beta} \right)\,.
\end{align*}
Now,
\begin{align*}
\|\nabla(\xi_R \zeta)\|_{p/(2-p)} \le & \| \zeta \nabla\xi_R\|_{p/(2-p)} + \|\xi_R \nabla \zeta\|_{p/(2-p)} \\
\le & R^{-(N+1)(p-p_c)/p} \|\nabla\xi\|_{p/(2-p)} +
\|\nabla\zeta\|_{p/(2-p)}\,,
\end{align*}
and thus
$$
\int_{\{|x|>r\}} u_M(t,x)\, dx \le C(r)\ \left( t^{1/p} +
R^{-(N+1)(p-p_c)/p}\ t^{1/p} + R^{-(\alpha-N\beta)/\beta} \right)\,.
$$
Letting $R\to\infty$ gives \eqref{FS.bound2}.
\end{proof}

We are now ready to construct the minimal very singular solution. By Lemma~\ref{lemma.fs2}, for any $t>0$, the sequence
$(u_{M}(t))_{M>0}$ is non-decreasing and uniformly bounded by
$\Gamma_{p,q}$. Thus, we can define
\begin{equation}\label{minimalVSS}
\overline{U}(t,x) := \sup_{M>0}\{ u_M(t,x)\} = \lim\limits_{M\to\infty}u_{M}(t,x)\,, \qquad (t,x)\in Q_\infty\,.
\end{equation}
Using once more Lemma~\ref{lemma.fs2}, we see that $\overline{U}(t)$
is radially symmetric for any $t>0$. Moreover, a first outcome of
Proposition~\ref{prop.compFG} and Lemma~\ref{lemma.fs1} is that
\begin{equation}
\overline{U} \leq u \;\;\text{ in }\;\; Q_\infty \;\;\text{ for any very singular solution }\;\; u \;\;\text{ to \eqref{eq1}.}
\end{equation}
It remains to show that $\overline{U}$ is a very singular solution to \eqref{eq1}.

\begin{proposition}\label{prop.minimalVSS}
The function $\overline{U}$ constructed in \eqref{minimalVSS} is a very singular solution to \eqref{eq1}. Moreover, $\overline{U}=U$, the latter being defined in Theorem~\ref{th.VSSunique}.
\end{proposition}

\begin{proof}
We first prove that $\overline{U}$ has the expected behavior as $t\to 0$. Let $r>0$. On the one hand, if $M>0$, we have $\overline{U}\ge u_M$ in $Q_\infty$ by \eqref{minimalVSS} and thus
\begin{equation*}
\liminf_{t\to 0} \int_{\{|x|\leq r\}} \overline{U}(t,x) \,dx \geq \lim_{t\to 0} \int_{\{|x|\leq r\}} u_{M}(t,x) \,dx =M\,,
\end{equation*}
from which the expected concentration \eqref{VSS1} of $\overline{U}$ at the
origin follows. On the other hand, we infer from the monotone convergence theorem and \eqref{FS.bound2} that
\begin{equation*}
\int_{\{|x|\geq r\}} \overline{U}(t,x) \,dx =
\lim\limits_{M\to\infty} \int_{\{|x|\geq r\}} u_{M}(t,x) \,dx \leq
C(r)\ t^{1/p}\,.
\end{equation*}
Letting $t\to 0$ gives the expected vanishing \eqref{VSS2} outside the origin.

Finally, it follows from Lemma~\ref{lemma.fs2}~(b) that $(u_M)_M$ is bounded in $L^\infty(\tau,\infty;W^{1,\infty}(\real^N))$ for any $\tau>0$. This property and Lemma~\ref{le.wp5} ensure the time equicontinuity of the family $(u_M)_M$ in $(\tau,\infty)\times\real^N$ for all $\tau>0$. We then deduce from the Arzel\`a-Ascoli theorem that $(u_M)_M$ is relatively compact in $C([\tau,T]\times K)$ for all compact subsets $K$ of $\real^N$ and $0<\tau< T$. Recalling \eqref{minimalVSS}, we conclude that $(u_{M})_M$ converges to $\overline{U}$ uniformly in compact subsets of $Q_\infty$. Consequently, thanks to the stability of viscosity solutions \cite[Theorem 6.1]{OS}, $\overline{U}$ is a viscosity solution to
\eqref{eq1}, and thus a very singular solution in the sense of
Definition~\ref{def.VSS}.

It remains to prove that $\overline{U}$ has a self-similar form which follows from the scale invariance of \eqref{eq1} and is now a standard step. Indeed, for $\lambda\in (0,\infty)$ and
$M\in(0,\infty)$, define a rescaled version of $u_M$ by
\begin{equation}\label{rescaling}
u_{M}^{\lambda}(t,x):=\lambda^{(p-q)/(2q-p)}u_{M}(\lambda
t,\lambda^{(q-p+1)/(2q-p)}x)\,, \qquad (t,x)\in Q_\infty.
\end{equation}
By straightforward calculations, we find that $u_{M}^{\lambda}$ is a
solution to \eqref{eq1}. To identify its initial trace, we consider $\psi\in BC(\real^N)$ and write
\begin{equation*}
\begin{split}
\int_{\real^N} u_{M}^{\lambda}(t,x)\ \psi(x)\,dx &= \lambda^{(p-q)/(2q-p)}\ \int_{\real^N} u_{M}(\lambda t,\lambda^{(q-p+1)/(2q-p)}x)\ \psi(x) \,dx\\
&= \lambda^{(N+1)(q_*-q)/(2q-p)}\ \int_{\real^N} u_{M}(\lambda
t,y)\ \psi(\lambda^{-(q-p+1)/(2q-p)}y) \,dy.
\end{split}
\end{equation*}
Letting $t\to 0$, we find that the initial condition of $u_{M}^{\lambda}$ is $\lambda^{(N+1)(q_*-q)/(2q-p)} M \delta_0$. By Theorem~\ref{th.uniqfund}, we obtain $u_{M}^{\lambda} = u_{\lambda^{(N+1)(q_*-q)/(2q-p)}M}$. We now pass to the limit as $M\to\infty$ and deduce from \eqref{minimalVSS} that
\begin{equation*}
\overline{U}(t,x) = \lambda^{(p-q)/(2q-p)}\ \overline{U}(\lambda t,\lambda^{(q-p+1)/(2q-p)}x), \quad (t,x)\in Q_\infty\,.
\end{equation*}
Therefore, $\overline{U}$ has a self-similar form and since it is obviously radially symmetric due to Lemma~\ref{lemma.fs2}~(a) and \eqref{minimalVSS}, we infer from Theorem~\ref{th.VSSunique} that $\overline{U}=U$.
\end{proof}

A further outcome of the above analysis is the following result
which is a straightforward consequence of Lemma~\ref{lemma.fs1},
\eqref{minimalVSS}, and Proposition~\ref{prop.minimalVSS}.

\begin{corollary}\label{cor.nonumber}
If $u$ is a very singular supersolution to \eqref{eq1} in $Q_\infty$ such that
$$
u\in C(Q_{\infty}) \ {\rm and} \ u(t,x)\leq\Gamma_{p,q}(|x|), \qquad
(t,x)\in Q_\infty\,,
$$
then
$$
U(t,x)\leq u(t,x), \qquad (t,x)\in Q_\infty.
$$
\end{corollary}

\subsection{The maximal very singular solution}\label{sec4max}
We begin with the following general result for very singular subsolutions to \eqref{eq1}.

\begin{proposition}\label{prop.maxVSS}
Let $u$ be a very singular subsolution to \eqref{eq1}. Then there
exists a very singular solution $\overline{u}$ such that
$u\leq\overline{u}$ in $Q_\infty$.
\end{proposition}

\begin{proof}
Fix $\tau>0$ and let $u^\tau$ be the solution to \eqref{eq1} in $(\tau,\infty)\times\real^N$ with initial condition $u^{\tau}(\tau)=u(\tau)$. The comparison principle and Proposition~\ref{prop.compFG} then ensure that
\begin{equation}
u(t,x) \le u^{\tau}(t,x) \le \Gamma_{p,q}(|x|)\,, \qquad (t,x)\in (\tau,\infty)\times\real^N\,. \label{d24a}
\end{equation}

Moreover, the function $u(\tau)$ satisfies \eqref{wp1} and \eqref{wp2}  (with $\kappa=\gamma$) by Proposition~\ref{prop.compFG}  and it follows from Proposition~\ref{pr.wp2} that, for $t>\tau$,
\begin{equation}\label{est1.tau}
(t-\tau)^{\alpha-N\beta} \|u^{\tau}(t)\|_{1} + (t-\tau)^{\alpha}\|u^{\tau}\|_{\infty}\leq K_\gamma,
\end{equation}
and
\begin{equation}\label{est2.tau}
(t-\tau)^{\alpha+\beta} \|\nabla u^{\tau}(t)\|_{\infty}\leq K_\gamma.
\end{equation}
We also notice that, if $0<\tau_1<\tau_2$, the inequality \eqref{d24a} implies that $u^{\tau_2}(\tau_2)=u(\tau_2)\le u^{\tau_1}(\tau_2)$, whence
\begin{equation}
u^{\tau_1}(t,x)\geq u^{\tau_2}(t,x)\,, \qquad (t,x)\in(\tau_2,\infty)\times\real^N\,, \label{d25a}
\end{equation}
by the comparison principle. Owing to \eqref{d24a}, \eqref{est1.tau}, and \eqref{d25a}, we may define the pointwise limit
\begin{equation}\label{est3.tau}
W(t,x):=\sup\limits_{\tau\in (0,t/2)} \{u^{\tau}(t,x)\} = \lim\limits_{\tau\to 0} u^{\tau}(t,x), \quad (t,x)\in Q_\infty.
\end{equation}

The remainder of the proof is devoted to proving that $W$ is a
very singular solution to \eqref{eq1} in $Q_\infty$. Consider $n\ge 1$. By \eqref{est1.tau} and \eqref{est2.tau}, the family $\{u^\tau\ :\ \tau\in (0,1/2n) \}$ is bounded in $L^\infty(1/n,n;W^{1,\infty}(\real^N))$ which allows us to apply Lemma~\ref{le.wp5} and deduce from the Arzel\`a-Ascoli theorem that $\{u^\tau\ :\ \tau\in (0,1/2n) \}$ is relatively compact in $C((1/n,n)\times B_n(0))$. Consequently, the pointwise convergence \eqref{est3.tau} of $(u^\tau)_\tau$ to $W$ can be improved to convergence in $C((1/n,n)\times B_n(0))$ for all $n\ge 1$, from which we deduce that $W$ is a viscosity solution to \eqref{eq1} in $Q_\infty$ by the stability of viscosity solutions \cite[Theorem~6.1]{OS}. We may also use this convergence to pass to the limit as $\tau\to 0$ in \eqref{d24a} and obtain
\begin{equation}
u(t,x) \le W(t,x) \le \Gamma_{p,q}(|x|)\,, \qquad (t,x)\in Q_\infty\,. \label{d27}
\end{equation}

It remains to prove that the function $W$ has the expected behavior as $t\to 0$. Since $u$ is a very singular subsolution to \eqref{eq1}, it satisfies \eqref{VSS1} and so does $W$ by \eqref{d27}. The study of the behavior of $W$ outside the origin requires more work. Let $\zeta\in C^{\infty}(\real^N)$ be such that $0\le\zeta\le 1$,
$$
\zeta(x)=1 \ \hbox{if} \ |x|\geq 1, \quad \zeta(x)=0 \
\hbox{if} \ |x|\leq\frac{1}{2}.
$$
Fix $r>0$ and define $\zeta_r(x) = \zeta(x/r)$ for $x\in\real^N$. It follows from \eqref{def.weak} that, for $t>0$ and $\tau\in (0,t/2)$,
\begin{equation}\label{interm8}
\begin{split}
\int_{\real^N} u^{\tau}(t,x)\ \zeta_r(x) \,dx \le & \int_{\real^N} u^{\tau}(\tau,x)\ \zeta_r(x) \,dx + \int_{\tau}^{t} \int_{\real^N} |\nabla u^{\tau}(s,x)|^{p-1}\ |\nabla\zeta_r(x)| \,dx\,ds\\
\le & \int_{\{|x|\ge r/2\}} u(\tau,x) \,dx \\
& + \int_{\tau}^{t} \int_{\{r/2<|x|<r\}} |\nabla u^{\tau}(s,x)|^{p-1}\ |\nabla \zeta_r(x)| \,dx\,ds,
\end{split}
\end{equation}
since $u^{\tau}(\tau)=u(\tau)$ by definition. On the one hand, Fatou's lemma and \eqref{est3.tau} give
\begin{equation}
\int_{\real^N} W(t,x)\ \zeta_r(x)\ dx \le \liminf_{\tau\to 0} \int_{\real^N} u^\tau(t,x)\ \zeta_r(x)\ dx\,. \label{d29}
\end{equation}
On the other hand, since $u$ is a very singular subsolution, we have
\begin{equation}
\lim_{\tau\to 0} \int_{\{|x|\ge r/2\}} u(\tau,x)\, dx =0\,, \label{d30}
\end{equation}
while \eqref{wp100}, \eqref{d24a}, and \eqref{est1.tau} give, for $s>\tau$,
\begin{align*}
|\nabla u^\tau(s,x)|^{p-1} \le & C \left( \left\| u^\tau\left( \frac{s+\tau}{2} \right) \right\|_\infty^{1/\alpha p} + (s-\tau)^{-1/p} \right)^{p-1}\ \left( u^\tau(s,x) \right)^{2(p-1)/p} \\
\le & C\ (s-\tau)^{-(p-1)/p}\ \Gamma_{p,q}(|x|)^{2(p-1)/p}\,.
\end{align*}
Thus
\begin{equation}
\int_{\tau}^{t} \int_{\{r/2<|x|<r\}} |\nabla u^{\tau}(s,x)|^{p-1}\ |\nabla\zeta_r(x)| \,dx \,ds \leq
C(r)\ t^{1/p} \|\nabla\zeta\|_{\infty}.\label{interm11}
\end{equation}
Combining \eqref{interm8}, \eqref{d29}, \eqref{d30}, and \eqref{interm11} leads us to
\begin{equation*}
\int_{\real^N} W(t,x)\ \zeta_r(x) \,dx \leq
C(r)\ t^{1/p}\ \|\nabla\zeta\|_{\infty}.
\end{equation*}
Using the properties of $\zeta$ and letting $t\to 0$ in the above inequality, we conclude that $W$ satisfies \eqref{VSS2}. Summarizing, we have established that $W$ is a very singular solution to \eqref{eq1} in $Q_\infty$ which lies above $u$ by \eqref{d27}.
\end{proof}

We are now ready to construct the maximal very singular solution to \eqref{eq1}. We denote the set of very singular solutions to \eqref{eq1} in $Q_\infty$ by $\cs$. Since $U\in\cs$, $\cs$ is non-empty and we may define
\begin{equation}\label{maximalVSS}
V(t,x) := \sup\limits_{u\in\cs}\{u(t,x)\}, \quad (t,x)\in Q_\infty.
\end{equation}
We prove next that $V$ is itself a very singular solution to \eqref{eq1}. We begin with the following bounds.

\begin{lemma}\label{est.maxVSS}
For $t>0$, we have
\begin{equation}\label{estmax1}
t^{\alpha}\ \|V(t)\|_{\infty} + t^{\alpha+\beta}\|\nabla
V(t)\|_{\infty} \leq K_\gamma,
\end{equation}
and
\begin{equation}\label{estmax2}
U(t,x)\leq V(t,x)\leq\Gamma_{p,q}(|x|), \quad x\in\real^N.
\end{equation}
\end{lemma}

\begin{proof}
Since $U\in\cs$, the inequality \eqref{estmax2} follows at once from \eqref{maximalVSS} and Proposition~\ref{prop.compFG}. We next deduce from \eqref{decayVSS} and \eqref{maximalVSS} that $\|V(t)\|_\infty \le K_\gamma\ t^{-\alpha}$ for $t>0$ while
\eqref{gradestVSS} and \eqref{maximalVSS} entail that, for any $x\in\real^N$, $y\in\real^N$, $u\in\cs$, and $t>0$,
$$
u(t,x) \leq u(t,y) + K_\gamma\ t^{-(\alpha+\beta)}\ |x-y|\leq
V(t,y) + K_\gamma\ t^{-(\alpha+\beta)}\ |x-y|.
$$
Hence, passing to the supremum over $u\in\cs$
$$
V(t,x)\leq V(t,y) + K_\gamma\ t^{-(\alpha+\beta)}\ |x-y|,
$$
and $V(t)$ is Lipschitz continuous for all $t>0$ with Lipschitz constant $K_\gamma\ t^{-(\alpha+\beta)}$. Consequently, $V(t)\in W^{1,\infty}(\real^N)$ and satisfies \eqref{estmax1}.
\end{proof}

We can now establish the main property of $V$.

\begin{lemma}\label{lem.Bardi}
$V$ is a very singular subsolution to \eqref{eq1}.
\end{lemma}

\begin{proof}
Since $V$ is the supremum of a family of viscosity solutions to \eqref{eq1} by \eqref{maximalVSS}, the fact that $V$ is a viscosity subsolution to \eqref{eq1} follows from \cite[Proposition V.2.11]{BCD}. The regularity $V(t)\in L^1(\real^N)\cap W^{1,\infty}(\real^N)$ for $t>0$ is a consequence of Lemma~\ref{est.maxVSS} and the integrability at infinity of $\Gamma_{p,q}$ (see Lemma~\ref{le.wp6}). Also, the concentrating property \eqref{VSS1} at the origin as $t\to 0$ follows at once from \eqref{maximalVSS} since $U\in\cs$. It remains to check that $V(t)$ vanishes outside the origin as $t\to 0$. For that purpose, let $r>0$ and $R>r$. Since the annulus $K(r,R):= \{x\in\real^N\ :\ r/2\le |x|\le R\}$ is compact, there is a finite number $l$ of points $(y_i)_{1\le i \le l}$ in $\real^N$ such that
\begin{equation}
K(r,R)\subset\bigcup\limits_{i=1}^{l} B_{r/8}(y_i). \label{d37}
\end{equation}
We infer from \eqref{locest.VSS} that, for any $1\leq i\leq l$, $\lambda>0$, $t>0$, and $u\in\cs$, we
have
$$
u(t,x)\leq \lambda
e^{A_{\lambda,r/4}t}\ \exp\left(\frac{1}{\sigma_{y_i,r/4}(x)}\right), \quad x\in B_{r/8}(y_i).
$$
The above estimate being valid for all $u\in\cs$ we conclude that, for any $1\leq i\leq l$, $\lambda>0$, and $t>0$,
\begin{equation}
V(t,x)\leq \lambda
e^{A_{\lambda,r/4}t}\ \exp\left(\frac{1}{\sigma_{y_i,r/4}(x)}\right), \quad x\in B_{r/8}(y_i). \label{d38}
\end{equation}
Recalling \eqref{d37}, we infer from \eqref{estmax2} and \eqref{d38} that, for $t>0$ and $\lambda>0$,
\begin{eqnarray*}
\int_{\{|x|\geq r\}} V(t,x)\,dx & = &\int_{K(r,R)} V(t,x) \,dx + \int_{\{|x|>R\}} V(t,x)\,dx\\
& \leq & \lambda\ e^{A_{\lambda,r/4}t}\ \sum\limits_{i=1}^{l} \int_{B_{r/8}(y_i)}\ \exp\left(\frac{1}{\sigma_{y_i,r/4}(x)} \right) \,dx + \int_{\{|x|>R\}} \Gamma_{p,q}(|x|) \,dx.
\end{eqnarray*}
Passing first to the limit $t\to 0$ and then $\lambda\to 0$ gives
$$
\limsup\limits_{t\to 0} \int_{\{|x|\geq r\}} V(t,x) \,dx\le  \int_{\{|x|>R\}} \Gamma_{p,q}(|x|) \,dx
$$
for all $R>r$. Thanks to Lemma~\ref{le.wp6}, the right-hand side of the above inequality converges to zero as $R\to\infty$, so that $V$ satisfies \eqref{VSS2} and the proof is complete.
\end{proof}

We are now in a position to identify $V$.

\begin{proposition}\label{prop.maximalVSS}
The function $V$ defined in \eqref{maximalVSS} is a very singular
solution in the sense of Definition~\ref{def.VSS}. Moreover, it is
radially symmetric and has self-similar form, thus coinciding with
the unique self-similar very singular solution $U$ given by Theorem~\ref{th.VSSunique}.
\end{proposition}

A straightforward consequence of Proposition~\ref{prop.maximalVSS}
is that $\cs=\{U\}$, which proves Theorem~\ref{th.VSSgeneral}.

\begin{proof}
It follows from Proposition~\ref{prop.maxVSS} and
Lemma~\ref{lem.Bardi} that there exists a very singular solution
$\overline{u}$ to \eqref{eq1} such that
$$
V(t,x)\leq\overline{u}(t,x) \quad \hbox{for} \ \hbox{all} \
(t,x)\in(0,\infty)\times\real^N.
$$
The definition \eqref{maximalVSS} of $V$ implies that
$V\equiv\overline{u}$ and thus $V$ is the maximal very singular
solution. The radial symmetry and self-similarity of $V$ then follow
from the scaling and rotational invariances of \eqref{eq1}.
\end{proof}

\subsection{A comparison principle}\label{sec4cc}

An interesting consequence of the uniqueness of the very singular
solutions to \eqref{eq1} is the following comparison principle for
the related elliptic equation
\begin{equation}\label{eq.ell}
-\Delta_{p}v+|\nabla v|^q-\a v-\b y\cdot\nabla v=0 \quad\text{ in }\quad \real^N\,.
\end{equation}

\begin{theorem}\label{th.comp}
Let $v_1$ be a viscosity subsolution and $v_2$ be a viscosity
supersolution to \eqref{eq.ell} in $\real^N$, such that
\begin{equation}
v_i\in L^1(\real^N)\cap W^{1,\infty}(\real^N)\,, \quad v_i\ge 0\,,
\quad v_i\not\equiv 0\, \quad i=1,2. \label{wp11}
\end{equation}
Assume that
\begin{equation}\label{int.est}
\lim\limits_{R\to\infty}\int_{\{|y|\geq R\}} v_i(y)|y|^{\a/\b-N}\,dy=0, \
i=1,2, \;\;\text{ and }\;\; v_2(y)\leq\Gamma_{p,q}(|y|), \quad y\in\real^N.
\end{equation}
Then $v_1(y)\leq f_U(y) \leq v_2(y)$ for all $y\in\real^N$, where
$f_U$ is the profile of the very singular solution $U$ to
\eqref{eq1}, see Theorem~\ref{th.VSSunique}.
\end{theorem}

Besides its interest in itself, this comparison principle will also
be useful to settle the asymptotic behavior in Section~\ref{sec5}.

\begin{proof}
For $i=1,2$, define
$$
u_i(t,x):=t^{-\a}v_i(xt^{-\b}), \quad (t,x)\in Q_\infty\,.
$$
It is then straightforward to check that $u_1$ is a subsolution and $u_2$ is a supersolution to \eqref{eq1} in $Q_\infty$. Moreover, we have
$$
u_i\in C(Q_{\infty})\;\;\text{ and }\;\; u_i(t)\in L^1(\real^N)\cap
W^{1,\infty}(\real^N), \quad t>0\,, \quad i=1,2.
$$
On the one hand, for $i=1,2$ and any $r>0$, we have
\begin{equation*}
\begin{split}
\int_{\{|x|\geq r\}}u_i(t,x)\,dx&=t^{N\b-\a}\int_{\{|y|\geq
rt^{-\b}\}}v_i(y)|y|^{\a/\b-N}|y|^{N-\a/\b}\,dy\\&\leq
r^{N-\a/\b}\int_{\{|y|\geq rt^{-\b}\}}v_i(y)|y|^{\a/\b-N}\,dy,
\end{split}
\end{equation*}
which tends to 0 as $t\to0$ by \eqref{int.est}. On the other hand,
since $v_i\not\equiv 0$, there is $r_0>0$ sufficiently large such
that
$$
\int_{B_{r_0}(0)} v_i(y)\, dy >0\,, \quad i=1,2\,.
$$
Consequently, for $t>0$ sufficiently small ($t\in \left( 0,(r/r_0)^{1/\b} \right)$), we have
$$
\int_{\{|x|\leq r\}}u_i(t,x)\,dx=t^{N\b-\a}\int_{\{|y|\leq
rt^{-\b}\}}v_i(y)\,dy>t^{N\b-\a}\int_{B_{r_0}(0)}v_i(y)\,dy,
$$
which tends to $+\infty$ as $t\to 0$, since $N\b-\a<0$. It follows
that $u_1$ is a very singular subsolution to \eqref{eq1} and $u_2$
is a very singular supersolution to \eqref{eq1}. Furthermore
$$
u_2(t,x)=t^{-\a}v_2(xt^{-\b})\leq\gamma|x|^{-\a/\b}=\Gamma_{p,q}(|x|),
$$
for any $(t,x)\in Q_\infty$. By Theorem~\ref{th.VSSgeneral},  Corollary~\ref{cor.nonumber}, and
Proposition~\ref{prop.maxVSS}, we obtain
$$
u_1\leq U\leq u_2 \quad \hbox{in} \quad Q_{\infty}.
$$
We reach the conclusion by going back to the original variables.
\end{proof}

\section{Convergence to self-similarity}\label{sec5}

With all the preparations done in the previous sections, we
are now ready to prove the main result about asymptotic convergence.
The proof will be divided into several steps.

\begin{proof}[Proof of Theorem \ref{th.asympt}] Let us first notice
that the condition \eqref{wp111} implies that $R(u_0)<\infty$, where
$R(u_0)$ is defined in \eqref{radius.initial}. Moreover, there
exists a sufficiently large constant $\kappa>0$ such that
$$
u_0(x)\leq\kappa|x|^{-\a/\b} \quad \hbox{for any} \quad
x\in\real^N.
$$

\noindent \textbf{Step~1. Self-similar variables.} In a first step,
we pass to self-similar variables and define the new variables $(s,y)$ and function $v$ by
\begin{equation}\label{self.fct}
u(t,x)=:(1+t)^{-\a}v(s,y), \quad s:=\ln(1+t), \ y:=x(1+t)^{-\b}.
\end{equation}
Then $v$ solves the equation
\begin{equation}\label{eq.self}
\partial_{s}v-\Delta_{p}v+|\nabla v|^q-\a v-\b y\cdot\nabla v=0,
\quad (s,y)\in Q_{\infty},
\end{equation}
with initial condition $v(0)=u_0$ in $\real^N$.

\medskip

\noindent \textbf{Step~2. Estimates for $v$.} Starting from the
estimates established for $u$, we can deduce estimates for
$v$ in similar norms as follows. First, recalling the homogeneity of $\Gamma_{p,q}$,
we deduce from \eqref{wp7} that
\begin{equation}\label{FG.boundv}
\begin{split}
v(s,y)&=e^{\a s}u(e^s-1,ye^{\b s})\leq e^{\a s}\Gamma_{p,q}(|y|e^{\b s}-R(u_0))\\&\leq\Gamma_{p,q}(|y|-R(u_0)e^{-\b s})
\end{split}
\end{equation}
for any $(s,y)\in Q_{\infty}$. Then, the estimates \eqref{wp3}
can be easily transformed into the following ones for $v$:
\begin{eqnarray}
\|v(s)\|_{1} + \|v(s)\|_{\infty} + \|\nabla
v(s)\|_{\infty} &\leq& \left[ \left(\frac{e^s}{e^s-1}\right)^{\a-N\b} + \left(\frac{e^s}{e^s-1}\right)^{\a} + \left(\frac{e^s}{e^s-1}\right)^{\a+\b} \right]K_{\kappa} \nonumber\\
&\leq& 6K_{\kappa},\label{wp3.v}
\end{eqnarray}
for any $s>\ln{2}>0$, where $K_{\kappa}$ is the constant in
\eqref{wp3}. Finally, the pointwise upper bound \eqref{est.point}
reads
\begin{equation}
|y|^{\a/\b} v(s,y) \le C\ \left[ \sup_{|z|\ge |y| e^{\b s}/4}\left\{ u_0(z) |z|^{\a/\b} \right\} + |y|^{-1/\b} \right] \label{pompee}
\end{equation}
for $(s,y)\in (0,\infty)\times (\real^N\setminus\{0\})$.

\medskip

\noindent \textbf{Step~3. Lower bound for $v$.} We infer from
\cite[Proposition~1.8]{IL1} that $u(t,x)>0$ for $(t,x)\in Q_\infty$.
In particular, $u(1,0)>0$ and, since $u(1,\cdot)\in C(\real^N)$,
there is $m_0>0$ such that
\begin{equation}
u(1,x)\ge m_0\,, \qquad x\in B_1(0)\,. \label{pompon}
\end{equation}
Next, according to the analysis performed in \cite{IL2} (in
particular, Lemma~2.1, Lemma~2.8, Lemma~2.10, Proposition~2.11, and
Proposition~2.16 therein), there exists $a_*>0$ such that, for $a\in
(0,a_*)$, the maximal solution $g_a$ defined on $[0,R_m(a))$ to the
Cauchy problem
\begin{equation}\label{IVPa}
\left\{
\begin{array}{l}
(|g_a'|^{p-2}g_a')'(r)+\displaystyle{\frac{N-1}{r}} (|g_a'|^{p-2}g_a')(r)+\a g_a(r)+\b rg_a'(r)-|g_a'(r)|^q=0 , \\
 \\
g_a(0)=a, \ g_a'(0)=0,
\end{array}\right.
\end{equation}
has the following properties: there is $R(a)\in (0,R_m(a))$ such that
\begin{equation}
0 < g_a(r) \le a \;\;\text{ for }\;\; r\in [0,R(a))\,, \quad g_a(R(a))=0\,, \;\;\text{ and }\;\; g_a'(R(a))<0\,. \label{cesar}
\end{equation}
Introducing
$$
G_{a,\lambda}(t,x) := \left\{
\begin{array}{lcl}
\lambda^{p/(2-p)}\ t^{-\a}\ g_a\left( \lambda |x| t^{-\b} \right) & \text{ if } & |x|\in \left[ 0, R(a) t^\b/\lambda \right]\,, \\
0 & \text{ if } & |x| \ge R(a) t^\b/\lambda\,,
\end{array}
\right.
$$
for $(a,\lambda) \in (0,a_*)\times (0,1)$, the properties of $g_a$ guarantee that $G_{a,\lambda}\in C(Q_\infty)$ and is a subsolution to \eqref{eq1} in $Q_\infty$ (it can be interpreted locally as the maximum of two subsolutions to \eqref{eq1} in $Q_\infty$, namely the zero function and $(t,x)\longmapsto \lambda^{p/(2-p)}\ t^{-\a}\ g_a\left( \lambda |x| t^{-\b} \right)$). \\
Now, we set
$$
a_0:= \frac{a_*}{2} \in (0,a_*)\,, \quad t_0 :=\frac{1}{R(a_0)^p} \left( \frac{m_0}{a_0} \right)^{2-p}\,, \quad \lambda_0 := R(a_0)\ t_0^\b\,,
$$
and observe that, if $x\in B_{R(a_0) t_0^\beta / \lambda_0}(0) = B_1(0)$, then \eqref{pompon} implies that
$$
G_{a_0,\lambda_0}(t_0,x) \le a_0 \lambda_0^{p/(2-p)} t_0^{-\a} = a_0 \left( \lambda_0 t_0^{-\b} \right)^{p/(2-p)} t_0^{1/(2-p)} = m_0 \le u(1,x)\,.
$$
Since $u(1,x)> 0 = G_{a_0,\lambda_0}(t_0,x)$ if $x\not\in B_{R(a_0) t_0^\beta / \lambda_0}(0)$, we have $u(1,x) \ge G_{a_0,\lambda_0}(t_0,x)$ for all $x\in\real^N$ and the comparison principle entails that
\begin{equation*}
u(t+1,x) \ge G_{a_0,\lambda_0}(t+t_0,x)\,, \qquad (t,x)\in Q_\infty\,.
\end{equation*}
In particular, for $t>0$ and $x\in B_{R(a_0) (t+t_0)^\beta / \lambda_0}(0)$,
$$
u(t+1,x) \ge \lambda_0^{p/(2-p)}\ (t+t_0)^{-\a}\ g_{a_0}\left( \lambda_0 |x| (t+t_0)^{-\b} \right) \,.
$$
In terms of $v$, the previous lower bound reads
\begin{equation}
v(s,y) \ge \lambda_0^{p/(2-p)} \left( \frac{e^s}{e^s-2+t_0} \right)^\a g_{a_0}\left( \lambda_0 |y| \left( \frac{e^s}{e^s-2+t_0} \right)^\b \right)\,, \label{modeste}
\end{equation}
for $s>\ln{2}$ and $|y|\le (R(a_0)/\lambda_0) ((e^s-2+t_0)
e^{-s})^\b$.

\medskip

\noindent \textbf{Step~4. Half-relaxed limits.} To complete the
proof of the convergence, we introduce the half-relaxed limits
\cite{Bl94}, in a similar way as it has been previously used in
papers on large-time behavior, see \cite{ILV, Ro} for instance. We thus define
\begin{equation*}
\tilde{w}_{*}(s,y):=\liminf\limits_{(\sigma,z,\e)\to(s,y,0)}v\left(\frac{\sigma}{\e},z\right),
\quad
\tilde{w}^{*}(s,y):=\limsup\limits_{(\sigma,z,\e)\to(s,y,0)}v\left(\frac{\sigma}{\e},z\right)
\end{equation*}
for $(s,y)\in Q_\infty$. It is a standard fact that $\tilde{w}_{*}$ and $\tilde{w}^{*}$ do not depend on $s>0$, so that we can define
$$
w_{*}(y):=\tilde{w}_{*}(1,y)=\tilde{w}_{*}(s,y), \quad
w^{*}(y):=\tilde{w}^{*}(1,y)=\tilde{w}^{*}(s,y), \quad s>0\,.
$$
In addition, it follows from \cite[Th\'eor\`eme~4.1]{Bl94} that
$w_{*}$ is a viscosity supersolution and $w^{*}$ is a viscosity
subsolution to the stationary equation associated to
\eqref{eq.self}, that is, the elliptic equation \eqref{eq.ell}.
Moreover, the definition of $w_*$ and $w^*$ and \eqref{modeste}
ensure that
\begin{equation}
w_*\leq w^* \;\;\text{ in }\;\; \real^N \;\;\text{ and }\;\; \lambda_0^{p/(2-p)}\ g_{a_0}(\lambda_0 |y|) \le w_*(y) \;\;\text{ for }\;\; y\in B_{R(a_0)/\lambda_0}(0)\,. \label{cesar2}
\end{equation}
An obvious consequence of \eqref{cesar} and \eqref{cesar2} is that
$w_*$ and $w^*$ are both not identically equal to zero.

Our aim now is to show that $w_*\equiv w^*$ with the help of
Theorem~\ref{th.comp}. In order to apply it, we translate the
estimates for $v$ in Step~2 above into estimates for $w_*$ and
$w^*$. We readily notice that \eqref{FG.boundv} implies
\begin{equation}\label{FG.boundw}
w_*(y)\leq w^*(y)\leq\Gamma_{p,q}(|y|), \quad \hbox{for} \
\hbox{any} \ y\in\real^N,
\end{equation}
and that \eqref{wp3.v} implies that
\begin{equation}\label{wp3.w}
w_*(y)\leq w^*(y)\leq6K_{\kappa}, \quad \|\nabla
w_{*}\|\leq6K_{\kappa}, \ \|\nabla w^{*}\|\leq6K_{\kappa}, \quad
y\in\real^N,
\end{equation}
whence $w_*$ and $w^*$ belong to the space $L^1(\real^N)\cap
W^{1,\infty}(\real^N)$. In addition, taking into account the
condition \eqref{wp111} on $u_0$, we deduce from \eqref{pompee} that
$$
w_*(y)\leq w^*(y)\leq C|y|^{-(\a+1)/\b}, \quad y\in\real^N\,.
$$
Consequently,
\begin{equation}\label{VSS1.w}
\int_{\{|y|\geq r\}}  \left( w_{*}(y) + w^{*}(y) \right) |y|^{\a/\b-N}\,dy\leq
C\int_{r}^{\infty}s^{-1/\b-1}\,ds=C r^{-1/\b},
\end{equation}
which converges to 0 as $r\to\infty$. Gathering \eqref{cesar2},
\eqref{FG.boundw}, \eqref{wp3.w}, and \eqref{VSS1.w}, we are in a
position to apply Theorem~\ref{th.comp} and conclude that $w^*\leq
f_U \le w_*$ in $\real^N$.

Recalling \eqref{cesar2}, we have established that $w_*\equiv
w^*=f_U$, which in turn implies that
$$
\lim_{\varepsilon\to 0} \sup_{y\in K}\left\{ \left| v\left( \frac{1}{\varepsilon},y \right) - f_U(y) \right| \right\} = 0
$$
for any compact subset $K$ of $\real^N$ by \cite[Lemma V.1.9]{BCD}
or \cite[Lemme~4.1]{Bl94}. Owing to \eqref{FG.boundv} and the decay
of $f_U$ as $|x|\to\infty$ (see Theorem~\ref{th.VSSunique}), the
above convergence can be improved to the convergence of $v(s)$ to
$f_U$ in $L^\infty(\real^N)$ as $s\to\infty$. Going back to the
original variables gives \eqref{asympt} and ends the proof.
\end{proof}

\section*{Acknowledgements}

The research of R. I. is partially supported by the Spanish
project MTM2008-03176. Part of this work was done while Ph. L. enjoyed the hospitality and support of the Departamento de An\'alisis Matem\'atico of the Univ. de Valencia.

\bibliographystyle{plain}


\end{document}